\documentclass[a4paper,11pt]{article}

\usepackage[english]{babel}
\usepackage{amsmath,amssymb,amsthm,amsthm}
\usepackage{mathrsfs}
\usepackage{hyperref}
\usepackage{soul,xcolor}

\usepackage[noadjust]{cite} 
\usepackage{cancel}
\usepackage{dashbox}
\usepackage{xcolor}

\RequirePackage{geometry}
\usepackage{chemformula}
\def\R{\mathbb{R}}

\newtheorem{theorem}{Theorem}
\newtheorem{proposition}[theorem]{Proposition}
\newtheorem{corollary}[theorem]{Corollary}

\theoremstyle{definition}
\newtheorem{definition}[theorem]{Definition}
\newtheorem{remark}[theorem]{Remark}

\date{}

\usepackage{cancel}
\usepackage{dashbox}
\usepackage{multirow}
\usepackage{bigdelim}
\usepackage{makecell}
\usepackage{mathrsfs} 
\usepackage{framed}

\usepackage{amsmath}
\usepackage{amsfonts}
\usepackage{amsthm}
\usepackage{amssymb}
\usepackage{blkarray}
\usepackage{cases}
\usepackage{xfrac}
\usepackage{mathtools}
\usepackage{centernot}

\newtheorem{assumption}[theorem]{Assumption}

\usepackage{environ}

\usepackage{lipsum}

\definecolor{UNIMIblue}{rgb}{0.0,0.2, 0.4}  
\definecolor{ivory}{rgb}{1.0, 1.0, 0.94}
  
\definecolor{grey}{rgb}{0.50,0.50,0.5}
\definecolor{orange}{rgb}{1.0,0.8,0.55}
\definecolor{gelb}{rgb}{0.0,1.0,1.0}
\definecolor{weiss}{rgb}{1.0,1.0,1.0}
\definecolor{darkgreen}{rgb}{0,0.4,0}
\definecolor{rltred}{rgb}{0.75,0,0}
\definecolor{rltgreen}{rgb}{0,0.5,0}
\definecolor{rltblue}{rgb}{0,0,0.75}
\definecolor{lavender}{rgb}{.7 , .7 , .8}
\definecolor{darkorange}{rgb}{1,0,0.0}
\definecolor{darkorange}{rgb}{0,0,1.0}

\title{Well-posedness of a reaction-diffusion model with stochastic dynamical boundary conditions}

\author{Mario Maurelli\thanks{Dipartimento di Matematica, Universit\`a di Pisa, Largo Bruno Pontecorvo 5, 56127 Pisa, Italy; email: mario.maurelli@unipi.it}, Daniela Morale\thanks{Dipartimento di Matematica, Universit\`a degli Studi di Milano, via Cesare Saldini 50, 20133 Milano, Italy, Italy; email: daniela.morale@unimi.it}, Stefania Ugolini\thanks{Dipartimento di Matematica, Universit\`a degli Studi di Milano, via Cesare Saldini 50, 20133 Milano, Italy; email: stefania.ugolini@unimi.it}}

\begin{document}

 	\maketitle
 
 \abstract{We study the well-posedness of a nonlinear reaction diffusion partial differential equation system on the half-line coupled with a stochastic dynamical boundary condition, a random system arising from the description of the chemical reaction of sulphur dioxide with calcium carbonate stones. The boundary condition is given by a Jacobi process, solution to a stochastic differential equation with a mean-reverting drift and a bounded diffusion coefficient. The main result is the global existence and the pathwise uniqueness of mild solutions. The proof relies on a splitting strategy, which allows to deal with the low regularity of the dynamical boundary condition.
 }
 
 \medskip
 
 \textbf{MSC2020}: 35R60, 60H15, 35K57, 60H30
 
 \textbf{Keywords}: stochastic dynamical boundary conditions, nonlinear reaction-diffusion PDEs, application to sulphation.

\section{Introduction}

This paper concerns the study of the well-posedness and regularity  of the solution of the following random system:
\begin{equation} \label{eq:deterministic_intro}
	\begin{split}
	\frac{\partial }{\partial t}  (\varphi s)   &=	\nabla  \cdot(\varphi  \nabla s )  -\lambda \varphi s c ,   \qquad  (0,\infty)\times (0,T)\\
	\partial_t c &= -\lambda \varphi s c,  
	\end{split}
\end{equation}
with $\lambda\in \mathbb R_+$,  
coupled with initial conditions $s(x,0)= s_0(x), \,
c(x,0)= c_0(x)$ for $x\in \mathbb R_+$, $\varphi=\varphi(c)$, where randomness is due to the coupling of \eqref{eq:deterministic_intro} with a stochastic dynamical boundary condition  for $s$, given by
$$
s(t,0)=\psi_t,  \quad t\in [0,T].
$$
More precisely, the idea is to consider as boundary condition the stochastic process $\psi_t$ given by the unique bounded solution of the following stochastic differential equation of It\^o type 
\begin{equation}\label{eq:stochastic_boundary_intro}
	d\psi_t = \alpha(\gamma-\psi_t) dt +\sigma \sqrt{\psi_t\left(\eta-\psi_t\right)}dW_t,
\end{equation}
where $\alpha,\gamma,\sigma$ and $\eta$ are positive parameters and $W\equiv\{W_t\}_{t\in [0,T]}$ is a Wiener process.

\medskip

The deterministic system \eqref{eq:deterministic_intro}, endowed with a constant boundary condition, has been introduced  in  \cite{2004_ADN} as an adimensionless model providing a quantitative description of the diffusion and the chemical action of sulphur dioxide (\ch{SO2}) on the porosity of calcium carbonate stones (\ch{CaCO3}); it is assumed that polluted air (and so the sulphur dioxide) diffuses through the pores of the stones and interacts with the surface of the pores. In \eqref{eq:deterministic_intro} $s$ stands for the porous concentration of \ch{SO2}, namely the concentration taken with respect to the volume of the pores, and $c$ for the local density of \ch{CaCO3}.   The   function $\varphi(c)$ is the porosity (volume of pores/total volume), which is supposed to depend only on the density of calcite$$
\varphi(c)=A+B c.$$ The porosity is related to the concentration $\rho_s$ of \ch{SO2} (fraction of total volume occupied by \ch{SO2}) by the relation $ 	\rho_s =  \varphi(c) s.$
For a detailed derivation of the system \eqref{eq:deterministic_intro}, the interested reader may refer to \cite{2004_ADN}, where the authors investigate the qualitative behavior of the solution on the half-line. 
To the best of our knowledge, all the dynamical systems considered so far describing the degradation of calcium carbonate due to the attack of sulphur dioxide are deterministic ones, with the exception of \cite{2019_BCFGN_CPAA}, where a new nonlinear differential system  for marble sulphation including the surface rugosity, a function obeying a damage-like equation which describes the microscopic variation of a surface with respect to a flat configuration, is introduced.  During the numerical simulations of the deterministic system a random surface rugosity on the boundary based on Weibull'statistics is considered, while here a dynamical model for the boundary condition is proposed. 

\medskip

The mathematical properties of the  system \eqref{eq:deterministic_intro} on the real line $\mathbb R$   have been  studied in \cite{2005_GN_NLA,2007_GN_CPAA}  in the case of  constant boundary conditions. In \cite{2008_GN_CPDE} the well-posedness of the model  is faced in the case $x\in (0,\infty)$, as in the present paper,  with a boundary condition $\psi$ for $s$   given by a non negative function with $W^{1,1}$ regularity in time; in the same paper  the fast reaction limit and the large time behaviour is also discussed.  See also \cite{2004_ADN} for the mathematical study of a generalization of the deterministic model, with Dirichlet boundary conditions for the unknown $s$ expressed in terms of an assigned space-time function. In \cite{2004_ADN,2005_GN_NLA} the boundary condition for $s$ is assumed to be constant, even though, at least as far as numerical analysis is concerned, the authors state that the same analysis may be performed in the case where $s(0,t)$ is a given pre-selected bounded measurable positive function. They prove   existence and uniqueness results and some regularity properties of
global solutions, extending existence and uniqueness results of local solutions, at least for smooth initial data as in \cite{1997_dAncona}. 

\medskip

Similar evolution problems with  stochastic dynamical boundary conditions has been studied with particular attention to the case in which at the boundary the solution acts as a white noise; see, for example, \cite{2002_AB, 2015_barbu_bonaccorsi_Tubaro_JMAA,2017_Bon_Zan,2006_Bon_Zig,2015_Brzezniak_russo,1993_daPrato_Zabczyk,2004_chueshov,2007_debussche_Fuhrman_Tessitore} and the literature therein.
Particular attention has been devoted to the case where, at the boundary, the solution acts as a white noise. For example, in \cite{ 2015_Brzezniak_russo,1993_daPrato_Zabczyk} the framework for the analysis of Poisson and heat equations excited by the white in time and/or space noise on the boundary is considered, and  existence and uniqueness  of weak solutions in the space of distributions are proven. Closer to our setting are the works in \cite{2006_Bon_Zig, 2004_chueshov}, where the solution satisfies a stochastic differential equation at the boundary; these works use an approach via semigroup theory to show existence and uniqueness of mild solutions; see also \cite{2015_barbu_bonaccorsi_Tubaro_JMAA}. The difference of the present work with respect to \cite{2006_Bon_Zig, 2004_chueshov} is the presence of the equation for $c$, which makes the system \eqref{eq:deterministic_intro} not strongly parabolic.
\medskip

In this paper, we introduce the stochastic boundary condition \eqref{eq:stochastic_boundary_intro} with the aim to study the existence, uniqueness and regularity of the solution of the random system \eqref{eq:deterministic_intro}-\eqref{eq:stochastic_boundary_intro}. The motivation of the choice of equation  \eqref{eq:stochastic_boundary_intro} derives from a statistical study of the source term $\psi$ giving  the concentration of sulfur dioxide in the air. 
Indeed, as observed in \cite{2023_Arceci_Giordano_Maurelli_Morale_Ugolini} for the time series of the \ch{SO2} concentration (in air) in the Milano area, it is clearly not realistic to consider such concentration as a constant;   evidence of a random evolution can be easily recognized. Within a preliminary study, the authors  propose the same Pearson process, solution of a mean reverting equation with bounded Jacobi noise as a model for the boundary function $\psi_t $, taking advantage of an estimation procedure based on the time series of real data, combined with a numerical simulation of the system \eqref{eq:deterministic_intro}-\eqref{eq:stochastic_boundary_intro}. 
This class of diffusion processes \eqref{eq:stochastic_boundary_intro} is well-studied both in the biological literature, where they are known as Wright-Fisher processes, and in the mathematical one known as Jacoby processes. See, e.g.\cite{2023_Grothaus} for a recent approach to these processes via Dirichlet forms. The Pearson processes are interesting both for the modeling and theoretical point of view, since they share the feature that under certain conditions upon the process parameters $\alpha,\gamma,\sigma,\eta $ it can be proved that, when it starts within $[0,\eta], \eta >0 $, it remains positive and bounded in the finite region $[0,\eta]$. Furthermore one can establish that \eqref{eq:stochastic_boundary_intro} admits an invariant probability distribution, which is a generalized Beta distribution. It is a suitable way to introduce a bounded noise, in spite of the fact that it is driven by a Wiener process $W$, which is a very important issue in the applications \cite{2013_donofrio,2017_donofrio_flandoli}.

\medskip

The main mathematical problem with the stochastic boundary condition \eqref{eq:stochastic_boundary_intro} is the fact that this condition is no longer in $W^{1,1}$, even tough it is more regular than the white noise. More precisely, the diffusion process at the boundary inherits the same regularity of the Wiener process $W$, that is $\psi$ has almost surely trajectories in $C^\beta$, for every $\beta\in(0,1/2)$, i.e. the class of  H\"older continuous functions of order less than one half. So the well-posedness of the PDE needs a carefully investigation.  Indeed,  randomness at the boundary is coupled with a system of nonlinear parabolic equations which is not strongly parabolic  and only some energy estimates are available for analysing it, as noticed in \cite{2004_ADN,2005_GN_NLA}. We stress that the strategy of the proof using in \cite{2008_GN_CPDE} cannot be adopted here, essentially because of the non differentiability of the boundary function  $\psi_t$ (see Remark \ref{rmk_Natalini1}).  

\medskip

Our approach is based on a splitting strategy, in which the solution $(s,c)$ of the system \eqref{eq:deterministic_intro}-\eqref{eq:stochastic_boundary_intro}  is seen as $(u+v,c),$ where from one side $u$ is solution to the heat equation coupled with the stochastic boundary condition \eqref{eq:stochastic_boundary_intro}, not depending upon the less regular function $c$; on the other side $(v,c)$ is a solution of a non linear system,  depending upon $u$ (in a functional way), but endowed with a deterministic constant boundary condition. This strategy allows us to compensate the irregularity of the stochastic boundary path with the regularity of $u$, coming from the heat equation. 
In particular, concerning $u$, the $W^{1,q}$ norm of $u$ can be controlled, in a pathwise sense, by the $C^\beta$ norm of $\psi$, for $q<1/(1-2\beta)$ (see Proposition \ref{prop:heat_bd}); this is a rather classical result but we give here a self-contained proof based on fractional Sobolev spaces and interpolation theory.  
 
As far as concerns the couple of variables $(v,c)$, the equation satisfied by $v$ is nonlinear and nonlocal, due to the dependence of $c$ on $v$ itself, involving also $u$ and $\partial_x u$. The assumptions of classical existence theorems for nonlinear scalar parabolic equations are not verified  because of the lack of $L^\infty$ estimate of the quantity $\partial_x c$, as mentioned in \cite{2004_ADN,2005_GN_NLA,2008_GN_CPDE}, where the authors  prove a global existence result for suitable weak solutions, based on the control of $\partial_x c$ in $L^2$. We partially use their results, but in our case we cannot apply the same strategy, due to the terms containing $u$ and $\partial_x u$. To control these terms, we use the $W^{1,q}$ bound on $u$ in Proposition \ref{prop:heat_bd}. With this argument, we get the $L^2$ bound on a linearized system, Proposition \ref{prop:main_bound_v}, and the a priori $L^2$ estimate on $s$, Proposition \ref{cor:a_priori_bd}, which represent the main novelty of our proof. Our principal result states the well-posedness for the nonlinear system \eqref{eq:deterministic_intro} under the hypothesis $\psi \in C^\beta([0,T])$ with $  \beta  \in (1/4 ,1/2) $.
The lower bound $\beta>1/4$ on the H\"older exponent for $\psi$ comes from the $L^2$ integrability condition required for $\partial_x c$, which in turn forces  the solution to the heat equation $u$ to be in $W^{1,2}$.

\medskip
  
To sum up, we extend the well-posedness result in \cite{2008_GN_CPDE} by adding a random dynamical boundary condition, given by the solution to an SDE with non Lipschitz diffusion coefficient. In particular, from the analytical point of view, we allow the boundary condition to have lower regularity with respect to the case in \cite{2008_GN_CPDE}, and precisely to be only $\beta$-H\"older continuous, with $\beta>1/4$. It would be interesting to investigate the role of an intermediate irregularity between the one typical of Brownian motion and the one of white noise. We plan to work on this direction in a future paper.
 
\medskip

The paper is organized as follows. 
In Section 2 we introduce the  PDE-ODE model with a stochastic dynamical boundary condition. We discuss and prove the principal properties of the boundary process, as boundness and H\"older continuity. The splitting strategy and our definition of mild solution are specified. The solution of the heat equation with the stochastic boundary condition above described is analysed in Section 3 with the aim to establish for it  continuity  results in terms of the $C^\beta $ norm of the stochastic boundary function. In Section 4 a study of a version  of  a linear equation for $\tilde{s}=\tilde{u}+\tilde{v}$ with prescribed regularity upon the coefficient is described. In particular the existence of  a  mild solution $\tilde{v}$, and of $\tilde{s}$ in $L^\infty([0,T],W^{1,p}(\mathbb R_+))$ is established. Such results are used in Section 5, where the main well-posedness achievements are illustrated.

\bigskip
\begin{paragraph}{Notations.} Throught the paper, for sake of simplicity  we have  introduced in the proofs  the following notations for the involved Banach  spaces $L^p_x:=L^p(\mathbb R_+)$ in space and $L_t^p(L_x^q):=L^p([0,T],L^q(\mathbb R_+))$ in time and space. Similar notation is considered for the fractional Sobolev spaces. Furthermore, given a function $f$ defined on $[0,T]\times \mathbb R_+$, we define the marginals as $f_t:=f(t,\cdot)$, for $t\in [0,T]$ and $f_x:=f( \cdot,x)$, for $x\in \mathbb R_+.$
\end{paragraph}

\begin{paragraph}{Acknowledgements.} We acknowledge support from Universit\`a degli Studi di Milano through the SEED grant `SciCult: Modellizzazione matematica e analisi SCIentifica per i beni CULTurali: previsione e prevenzione del degrado chimico e meccanico di pietre monumentali in ambienti outdoor'. Part of the work was done while M.M. was at Universit\`a degli Studi di Milano. We thank Marco Alessandro Fuhrman for proposing this class of problems, Stefano Bonaccorsi, Cecilia Cavaterra and Roberto Natalini for relevant discussions on the topic. MM, DM and SU are members of the Istituto Nazionale di Alta Matematica, group GNAMPA.
\end{paragraph}

\section{A PDE - ODE model with a  nonsmooth boundary condition}

Let us introduce and describe the initial boundary value problem under study.   Hence, let us consider the following reaction diffusion system  for the couple $(s,c)$

\begin{eqnarray} 
	\frac{\partial }{\partial t}  (\varphi s)   &=&	\nabla  \cdot(\varphi  \nabla s )  -\lambda \varphi s c ,   \qquad  (0,\infty)\times (0,T)\label{eq:det_nostro_s}\\
	\partial_t c &=& -\lambda \varphi s c , \label{eq:det_nostro_c} 
\end{eqnarray}
where $\lambda\in \mathbb R_+$  with non-negative initial conditions 
\begin{eqnarray} \label{eq:initial_condition}
	s(x,0)= s_0(x),\qquad 
	c(x,0)= c_0(x), \qquad x \in (0,\infty).
\end{eqnarray}

\medskip

\paragraph{ An example of stochastic dynamical boundary condition}
We study the case in which the  boundary function for the solution $s$ is random and it is given by a stochastic process
\begin{eqnarray} 
	s(0,t) &=&   \psi_t  , \qquad     \psi_t \in C^{\beta}, \quad 1/4 <\beta< 1/2.
	\label{eq:stoc_nostro_boundary_cond} 
\end{eqnarray}
The  main example we have in mind,  and our motivation to consider the above regularity assumptions on $\psi$, is a specific Pearson process $\{\psi_t\}_{t\in [0,T]}$, solution to the following mean reverting generalized Jacobi stochastic differential equation 
\begin{equation}\label{eq:jacobi} 
	\begin{split}
		d\psi_t &= \alpha(\gamma-\psi_t) dt +\sigma \sqrt{\psi_t\left(\eta-\psi_t\right)}dW_t,  \\
		\psi_0 &\in [0,\eta],
	\end{split}
\end{equation}
with $\alpha,\sigma,\gamma,\eta\in \mathbb R_+$, and $\gamma\le \eta$. Overall we have a PDE problem with a boundary condition which is not only stochastic but has less regularity. Indeed, the following properties hold.

\begin{proposition}\label{prop:stoch_properties} Let us consider equation \eqref{eq:jacobi}, with $\alpha,\sigma,\gamma,\eta\in \mathbb R_+$ and $\gamma\le \eta$. Assume that $\psi_0 \in [0,\eta]$.
	The solution of equation \eqref{eq:jacobi} exists (globally on $[0,T]$) and is pathwise unique.   Furthermore, let us suppose that  the following conditions upon the parameters hold
	\begin{eqnarray}\label{eq:coefficients_SDE}
		\alpha\gamma  &\geq& \frac{\sigma^2\eta}{2};  \qquad \qquad  
		\alpha(\eta-\gamma) \geq \frac{\sigma^2\eta}{2}.
	\end{eqnarray}  Then  for any $t\in (0,T], $ 
	\begin{equation}\label{eq:psi_limits}
		 \psi_t \in (0,\eta).
	\end{equation} 
	In particular, we have $\psi_t\in L^\infty(\mathbb R_+)$.
\end{proposition}
\begin{proof} We give a sketch of the proof. See, for example, \cite{1981_karlin_taylor}. We notice that the drift  coefficient  $\alpha_1(x)=\alpha(\gamma -x)$ is Lipschitz continuous for any $x\in \mathbb R$, while the diffusion  coefficient $\sigma_1(x)=\sigma\sqrt{x(\eta-x)}$ is only continuous for any $x\in [0,\eta]$. Furthermore, for any $x\in [0,\eta]$,  \begin{align*}
		|\alpha_1(x)|^2 & = |\alpha\gamma - \alpha x|^2  \leq \alpha^2(|\gamma|^2+|2\gamma x|+|x|^2) \leq C_1(\alpha,\gamma)\biggl(1+|x|^2\biggr)\\
		|\sigma_1(x)|^2 & = |\sigma\sqrt{x(\eta-x)}|^2 \leq |\sigma|^2( |\eta x| + |x|^2 ) \leq C_2(\sigma,\eta) (1+|x|^2)
	\end{align*}
	By classical results due to Skorokhod, we get that the solution exists \cite{1995_skorokhod}. The  pathwise  uniqueness can be proven by applying the Yamada-Watanabe uniqueness criterion in  \cite{1971_yamada_watanabe}, since
	$$
	|\sigma_1(x)-\sigma_1(y)|\leq \rho(|x-y|), \quad \forall x,y\in \R,\quad i=1,2,\dots,  
	$$ with $\rho(x)=\sqrt{x}$. Indeed, 
	\begin{align*}
		|\sigma_1(x)-\sigma_1(y)| &  \leq \sigma \sqrt{|\eta x - x^2 - \eta y + y^2|} \\
		& \leq \sigma \sqrt{|x-y|}\sqrt{|\eta-x-y|} \\
		& \leq C_3(\eta)\cdot |x-y|^{1/2}.
	\end{align*} 
	Hence, 	  the diffusion coefficient   $\sigma_1(x)$ is H\"older continuous with index $1/2$,   for any $x\in [0,\eta]$.
	
	\medskip
	
	In order to obtain the bound \eqref{eq:psi_limits} for the solution, we consider 
	the well-known Feller  classification of  the boundaries \cite{1981_karlin_taylor,2016_etheridge}. In the  case under study, given a little generalization of the Jacobi process case, we have that the scale function $	scale(x)$ and the speed density $m(x)$ are defined as 
	$$
	scale(x) = \frac{1}{x^p(\eta-x)^q}, \qquad \qquad 
	m(x) =  \frac{1}{\sigma^2} x^{p-1}(\eta-x)^{q-1},
	$$
	with $$	p = \frac{2\alpha\gamma}{\sigma^2\eta},\qquad\quad q = \frac{2\alpha(\eta-\gamma)}{\sigma^2\eta}.$$
	The boundary $x=0$ is an \textit{entrance boundary} if and only if $p\ge 1$  while   $x=\eta$ is an \textit{entrance boundary} if and only if $q \ge 1$ and these two condition are satisfied whenever \eqref{eq:coefficients_SDE} hold (see \cite{1981_karlin_taylor},\cite{2023_Arceci_Giordano_Morale_Ugolini}).
\end{proof}
	
	We also recall the classical H\"older regularity property of solutions to the SDE:

\begin{proposition}
	 Under the same assumptions of Proposition \ref{prop:stoch_properties}, the solution of equation \eqref{eq:jacobi} has trajectories almost surely  in $C^\beta([0,T]), \, \beta \in (0,1/2)$.
\end{proposition}
\begin{proof} If $\psi$ is a solution to \eqref{eq:jacobi}, then, for any $s,t \in [0,T]$, one may write, for $p > 2$
	\begin{equation}\label{eq:prop_hoelder}
		\left|\psi_t - \psi_s\right|^p \le \left| \int_s^t \alpha(\gamma-\psi_\tau) d\tau \right|^p + \left|\int_s^t  \sigma \sqrt{\psi_\tau\left(\eta-\psi_\tau\right)}dW_\tau\right|^p.  
	\end{equation}
	By Burkholder-Davis-Gundy inequality \cite{2000_Rogers_williams}(IV,42), and by the bound \eqref{eq:psi_limits} one obtained  for every $2 < p<\infty$, there exists $C>0$ such that, for every $s<t$,
	\begin{align*}
		\mathbb{E}\left[\left|\int_s^t \sigma \sqrt{\psi_\tau\left(\eta-\psi_\tau\right)} dW_\tau\right|^p\right] \le C \mathbb{E}\left[\left(\int_s^t \left|\sigma \sqrt{\psi_\tau\left(\eta-\psi_\tau\right)}\right|^2 dr\right)^{p/2}\right]\le C(\eta) \left| t-s\right|^{p/2}.
	\end{align*}
	Then from \eqref{eq:prop_hoelder}, we have
	$$
	\mathbb{E}\left[\left|\psi_t - \psi_s\right|^p\right] \le C(\alpha,\gamma)|t-s|^p+
	C(\eta) \left| t-s\right|^{p/2}\le C(\alpha,\gamma,\eta) \left| t-s\right|^{p/2}$$
	By Kolmogorov continuity criterion,  with $r_1=p$ and $1+r_2= p/2$,
	  the solution $\psi_t$  is $C^{\beta}$-H\"older continuous   for any $\beta\in (0,r_2/r_1)=(0,1/2 -1/p)$ and $p>2$, that is $\beta\in (0,1/2)$.
\end{proof}

\begin{remark}
  As said before, the Jacobi process solving \eqref{eq:jacobi} is the main motivation to consider $\psi$ with H\"older regularity in time. However, our analysis will be deterministic and will apply to sample paths of any processes with $C^\beta$ regularity in time, for $\beta>1/4$ (and satisfying the Assumption \ref{hp:main_result} below).
\end{remark}

\paragraph{A model for the porosity}

 From now on, unless otherwise stated, $\psi:[0,T]\to \R$ will be a given function.

The function $\varphi$ in system \eqref{eq:det_nostro_s}-\eqref{eq:det_nostro_c} refers to the porosity, which is the fraction of void volume with respect to the total one. Considering the literature on the subject, we assume (see, for example, \cite{2004_ADN})
\begin{equation}\label{eq:phi_c_AB}
	\varphi(c) = A+Bc.
\end{equation}
 with $A>0$, $B\neq 0$. By the scaling $s^\prime =|B|s$, $c'=|B|c/A$, $\varphi^\prime(c^\prime)=A\varphi(c)$, we can reduce ourselves to the case
\begin{equation}
    A=1,\quad B=\pm 1;\label{eq:AB_hp}
\end{equation}
hence we will assume \eqref{eq:AB_hp} from now on. In the specific case of marble sulphation the value $B=-1$ is more appropriate.

\begin{proposition}
	If equation \eqref{eq:phi_c_AB} holds, the system 
	\eqref{eq:det_nostro_s}-\eqref{eq:stoc_nostro_boundary_cond} becomes	
	\begin{eqnarray} 
		\partial_t    s   &=&	\partial_x^2 s + b_c(t,x) \partial _x s + \gamma_c(t,x) s(Bs- 1)  \label{eq:s2}\\
		\partial_t c &=& - \lambda s (1+B c) c  \label{eq:c2}
	\end{eqnarray}	
	with \begin{eqnarray} 
		b_c(t,x)&=&B\frac{\partial_x c}{(1 + Bc)} \label{eq:beta2}\\
		\gamma_c(t,x) &=& \lambda c   \label{eq:gamma2}
	\end{eqnarray}
\end{proposition}

\begin{proof}
	Equation \eqref{eq:s2} is derived in the following simple way
	\begin{eqnarray*}
		\partial_t  ((1+Bc) s)   &=&	\nabla  \cdot((1+Bc)\nabla s )  -\lambda (1+Bc) s c  \\
		\partial_t  s + s  B {\partial_t}   c + B c 	{\partial_t}   s &=&	\partial_x^2 s +  B \nabla  (c \nabla s)   -\lambda (1+Bc)   c s  \\
		(1+Bc)	 \partial_t   s     &=&	\partial_x^2 s +  B \partial_xc \partial_x s+ Bc \partial^2_x  s -\lambda c (1+Bc) s + B\lambda c (1+Bc)s^2  \\
		(1+Bc)	 \partial_t   s   &=&	(1+Bc)  \partial_x^2 s +  B \partial_x c \partial_x s   + \lambda c s  (1+Bc) (Bs-1) 
	\end{eqnarray*}
	By dividing by $(1+Bc)$ we have the result.
\end{proof}

\medskip
We want to prove an existence and uniqueness result  for  system 
\eqref{eq:s2}-\eqref{eq:gamma2}.

\medskip

\begin{definition}[Heat kernels]  Denoted by $G_1(x)=1/\sqrt{2\pi} e^{-x^2/2} $ the standard Gaussian density, let us introduce the heat kernel $G(t,x)$  for any $(t,x)\in [0,T]\times \mathbb R$ 
\begin{equation}\label{eq:heat_kernel}
	G(t,x) = t^{-1/2} G_1(t^{-1/2}x) = \frac{1}{\sqrt{4\pi t}}e^{-x^2/4t}.
\end{equation}
  On the half-line, we refer to the following kernel on $(t,x)\in [0,T]\times \mathbb R_+$
 \begin{equation}\label{eq:dirichlet_heat_kernel}
\bar{G}(t,x,y)=G(t,x-y)-G(x+y).
\end{equation}  with $G$ as in \eqref{eq:heat_kernel} as the \emph{Dirichlet heat kernel} on the half line.
 
\end{definition}
\begin{definition}[Convolution $*_D$]    We call \emph{convolution with respect to the Dirichlet heat kernel} and we denote it as  $\bar{G}(t,\cdot)*_Df $ the following operator, for any regular $f$ 
	\begin{align}\label{eq:K_conv}
  		\bar{G}(t,\cdot)*_D f \,(x)  = \int_0^\infty (G(t,x-y)-G(t,x+y))f(y) dy = \int_{\mathbb{R}} G(t,x-y)f^{odd}(y) dy,
  	\end{align}
  	where $f^{odd}(x) =1_{x>0}f(x) -1_{x<0}f(-x)$ is the odd extension of $f$.
\end{definition}
\begin{definition}[Bounded positive mild solution]\label{def:mild_solution_s_nonlin}
	We say that a couple $(s,c)$, with $s\in L^\infty([0,T],W^{1,2}(\mathbb R_+))\cap L^\infty\left([0,T]\times \mathbb R_+\right)$ and $c\in B_{b}\left([0,T]\times \mathbb R_+\right)$, the space of bounded Borel functions, is \emph{a bounded positive mild solution} of the nonlinear PDE \eqref{eq:s2}-\eqref{eq:gamma2}  with initial conditions 
 \eqref{eq:initial_condition} and boundary  condition \eqref{eq:stoc_nostro_boundary_cond} if the following happens
 \begin{itemize}
     \item[i)]
 for every $x \in \mathbb R_+$, the function $c(\cdot,x)$ solves \eqref{eq:c2} and \eqref{eq:gamma2}, that is $c$ is explicitly given by 
	\begin{equation}\label{eq:df_mild_solution_c}
	c(t,x) = \frac{c_0(x)}{\varphi(c_0(x))e^{\lambda \int_0^t s(x,\tau)d\tau} -Bc_0(x)},
	\end{equation}
\item[ii)] the function $s$ satisfies
	\begin{align*}
		0\le s(t,x)\le \eta \quad \text{for every }(t,x)\in [0,T]\times \mathbb R_+,
	\end{align*}
	and it is a $L^\infty([0,T],W^{1,2}(\mathbb R_+))$-mild solution of \eqref{eq:s2}, that is
	\begin{align}
	\begin{split}\label{eq:df_mild_solution_s}
		s(t,\cdot) &= -2\int^t_0\partial_x G(t-\tau,\cdot) \psi(\tau)d\tau + \overline{G}(t,\cdot)*_Ds_0\\
		&\quad +\int_0^t \overline{G}(t-\tau,\cdot)*_D\left(b_c(\tau,\cdot)\partial_x s(\tau,\cdot)+\gamma_c(\tau,\cdot) s(\tau,\cdot)\right)(Bs(\tau,\cdot)-1)) ds.
	\end{split}
	\end{align} 
    \end{itemize}
In equation \eqref{eq:df_mild_solution_s}	  $b_c$ and $\gamma_c$ are defined as   \eqref{eq:beta2} and \eqref{eq:gamma2}, respectively. 
\end{definition} 
 Let us remark  in the mild formulation \eqref{eq:df_mild_solution_s} the dependence upon the boundary condition $\psi$ and the initial condition $s_0$, as in \eqref{eq:initial_condition}  and  \eqref{eq:stoc_nostro_boundary_cond}.

\medskip

The idea to show well-posedness of \eqref{eq:s2}-\eqref{eq:gamma2} is to consider a splitting strategy in order to divide the problem into two sub-problems: in the first one we consider the heat equation with the stochastic, H\"older continuous boundary condition, in the second problem we consider a non linear PDE with zero boundary condition. The initial conditions of the original problems are inherited by the latter system.

\paragraph{A splitting strategy}
Let us consider the following splitting of the solution $s$ of Equation \eqref{eq:s2}
$$s=u+v,$$
where the function 
$u$ is the solution of the following  \emph{heat equation with a  non-smooth boundary condition}
\begin{eqnarray} 
	{\partial_t}  u      &=&	\partial_x^2 u   \label{eq:u}\\
	u(x,0)&=& 0, \qquad x \in (0,\infty)\label{eq:u_initial_condition}\\
	u(0,t) &=&   \psi_t  , \qquad     \psi_t \in C^{\beta} ([0,T]), \quad\beta< 1/2\label{eq:u_boundary_condition}
\end{eqnarray} 
and the function $v$ is the solution of the following \emph{deterministic reaction diffusion problem} 

\begin{eqnarray} \label{eq:splitted_v_c_1}
	\frac{\partial }{\partial t}    v  &=&	\partial_x^2 v + b(t,x) \partial _x v- \gamma(t,x) v  + b(t,x) \partial_x u - \gamma(t,x) u +\gamma(t,x)B (u+v)^2 \\
	\partial_t c &=& - \lambda s (A+Bc) c  \label{eq:splitted_v_c_2}
\end{eqnarray}
with \begin{eqnarray*}
	b(t,x)&=&B\frac{\partial_x c}{(A+Bc)}, 
    \\
	\gamma(t,x) &=& \lambda c.   
\end{eqnarray*}
The system \eqref{eq:splitted_v_c_1}-\eqref{eq:splitted_v_c_2} has the following initial and boundary conditions
\begin{eqnarray*} 
	c(x,0)&=& c_0(x), \qquad x \in (0,\infty);\\ 
	v(x,0) &=&   s_0(x),\qquad x \in (0,\infty); \\
	v(0,t)&=&0, \qquad \qquad t\in (0,T].
\end{eqnarray*}

\begin{remark}
Let us note that the system for $u$ \eqref{eq:u}-\eqref{eq:u_boundary_condition} is completely autonomous: it does not depend on the function $c$ of the original  system \eqref{eq:s2}-\eqref{eq:gamma2}. Differently, the system for $v$ is coupled with that of \eqref{eq:u}-\eqref{eq:u_boundary_condition}.	
\end{remark}

 Note that the analysis carried out in Section 3 and 4 and all the estimates are valid pathwise, that is for any fixed realization of the boundary process. At the end of Section 5 we discuss the problem of measurability.

\section{ Heat equation with 
 nonsmooth boundary condition: some bounds}

We first study the regularity of the system   \eqref{eq:u}-\eqref{eq:u_boundary_condition},  where $\psi:[0,T]\to \R$ is a given function (with suitable assumptions); though the regularity result for this system is rather classical, we give a self-contained argument. Let us recall that equation \eqref{eq:u} admits an explicit solution given in term of the fundamental heat solution. 
 We recall that a solution of the heat equation for $(t,x) \in \mathbb R_+ \times \mathbb R_+$ with initial condition $u(0,x),$ and inhomogeneous Dirichlet boundary condition $u(t,0)$ may be represented as \begin{equation}\label{eq:heat_solution_dirichlet}
u(t,x)=\int_{\mathbb R_+} \overline{G}(t,x,y)u(0,y)dy+ \int_{\mathbb R_+}\left[\partial_y\bar{G}(t-\tau,x,y)u(\tau,0)\right]_{|y=0}d\tau,\end{equation}
where $\bar{G}$ is  defined by  \eqref{eq:dirichlet_heat_kernel}. 
The interest reader may refer to 
  \cite{1984_cannon} for a detailed study. From the expression \eqref{eq:heat_solution_dirichlet}  we obtain an explicit solution of the initial-boundary problem \eqref{eq:u}-\eqref{eq:u_boundary_condition} by taking into account initial and boundary conditions \eqref{eq:u_boundary_condition} and \eqref{eq:u_initial_condition} as it follows 
\begin{align}
	u(t,x) = -2 \int_0^t {\partial_x G}(t-\tau,x)\psi( \tau)d\tau. \label{eq:sol_heat}
\end{align}
From \eqref{eq:heat_kernel}, the space derivatives of the heat kernel $G$ are given by
\begin{eqnarray}
	{\partial_x G}(t,x) &=& t^{-1} G_1^\prime(t^{-1/2}x) = -\frac{1}{4\sqrt{\pi  }}t^{-3/2}x e^{-x^2/4t}, \label{eq:partial_1_heat_kernel}\\
	&\nonumber\\
	{\partial^2_x G}(t,x) &=& t^{-3/2}G_1^{\prime\prime}(t^{-1/2}x) = \frac{1}{4\sqrt{\pi  }}\left( \frac{1}{2}t^{-1}x^2- 1 \right) t^{-3/2} e^{-x^2/4t}.\label{eq:partial_2_heat_kernel} 
\end{eqnarray}   
Hence,   for $p\in [1,\infty)$, there exist some constant $c_p,c_p^\prime \in \mathbb R_+$ such that, for any $t\in [0,T]$
\begin{align}\label{eq:heat_bd}
	\begin{aligned}
		\|G(t,\cdot)\|_{L^p(\mathbb R_+)} &= c_p \, t^{-(1-1/p)/2},\\
		\|\partial_x G(t,\cdot)\|_{L^p(\mathbb R_+)} &= c_p^\prime\, t^{-(2-1/p)/2}.
	\end{aligned}
\end{align}
\begin{proposition}  
	For any $\psi$ bounded Borel function, the solution $u$ given by \eqref{eq:u}-\eqref{eq:u_boundary_condition} is such that  pathwise $u\in C^\infty [(0,T)\times (0,\infty)]$ and $u$ and its derivatives admit continuous extensions to $[0,T)\times (0,\infty)$.
	\label{prop:first_properties_u} 
\end{proposition}\begin{proof} The statement follows from the fact that the  kernel $G$ in \eqref{eq:heat_kernel} is $C^\infty$ on $(0,T)\times (0,\infty)$ and $G$ and its derivatives admit continuous extensions to $[0,T)\times(0,\infty)$.
\end{proof}

The first step is to establish a bound for the norm of $u$ in a fractional Sobolev space, uniformly in time. Hence we recall here the definition of such spaces.

\begin{definition}[The fractional Sobolev spaces] 
	We consider the fractional Sobolev space $ {W}^{\alpha,p}$ upon $\mathbb R$ for $1\le p\le \infty$  and $0<\alpha<1$ as follows \cite{1978_triebel,2012_dinezza_palatucci_valdinoci} 
	\begin{equation*}
	{W}^{\alpha,p}(\mathbb R) = \left\{ f\in L^p(\mathbb R):  \frac{| f(x)- f(y)| }{|x-y|^{1/p+\alpha}}\in L^p(\mathbb R)\right\},
	\end{equation*} an intermediate Banach space between $L^p(\mathbb R)$ and $W^{1,p}(\mathbb R)$, endowed with the natural
	norm
	\begin{align*}
	\|f\|_{W^{\alpha,p}(\mathbb R) }^p= \|f\|_{L^p}^p + [f]_{W^{\alpha,p}(\mathbb R) }^p,
	\end{align*}
	where
	\begin{align*}
	[f]_{W^{\alpha,p}(\mathbb R) }^p = \iint \frac{|f(x)-f(y)|^p}{|x-y|^{1+\alpha p}} dxdy
	\end{align*} 
	is the so-called Gagliardo (semi)norm of $f$.
	Furthermore, one can define the fractional Sobolev space with exponent $k+\alpha$, with $k \in \mathbb N$
	\begin{equation*}
	{W}^{k+\alpha,p}(\mathbb R) = \left\{ f\in W^{k,p}(\mathbb R):  \partial^k f\in W^{\alpha,p}(\mathbb R)\right\},
	\end{equation*} endowed with the norm
	\begin{equation}\label{eq:def_norm_sobolev_fractional}
	\|f\|_{W^{k+\alpha,p}(\mathbb R) }^p= \|f\|_{L^p}^p + \|\partial^k f\|_{L^p}^p+[\partial^k f]_{W^{\alpha,p}(\mathbb R) }^p.
	\end{equation}
	In the case $\alpha=0$ the previous space is the classical Sobolev space. The previous definitions are extended also to $\mathbb R_+$.
	
\end{definition} 

\begin{proposition}\label{prop:bound_u_W_alpha,p}
	Let $u$ be the solution of the system \eqref{eq:u}-\eqref{eq:u_boundary_condition}, with the boundary condition such that $\psi\in L^\infty([0,T])$. Then, for any $ 0<\alpha<1$ and $1< p<1/\alpha$, and  for a constant $C$ not depending upon $\psi$, we have  
	\begin{align}
		\sup_{t\in [0,T]}\|u(t,\cdot)\|_{W^{\alpha,p}(\mathbb R_+)}\le C\|\psi\|_{L^\infty([0,T])}.\label{eq:bound_u_1}
	\end{align} 
\end{proposition}
\begin{proof}
	By Proposition \ref{prop:first_properties_u} we already have some continuous properties   far from the boundary. Now we study the behaviour of the solution  near the boundary $x=0$. 
	
	From \eqref{eq:partial_1_heat_kernel}, by considering the self similarity rescaling   $\xi = x/\sqrt{t-\tau}$, we analyse carefully the behaviour of $u$ both close and far from the origin. For some $c>0$ and a constant $C$, that may be different from line to line, one may obtain from \eqref{eq:sol_heat} the following bound for $u$, for any $(t,x)\in [0,T]\times \mathbb R_+$
	\begin{equation}
    \begin{split}
		|u(t,x)| &\le 2\|\psi\|_{L^\infty([0,T])} \int_0^t \left|{\partial_x G}(t-\tau,x)\right| d\tau  \\
		&= C\|\psi\|_{L^\infty([0,T])}  \int_{x/\sqrt{t}}^\infty \frac{\xi^2}{x^2} G_1^\prime(\xi) \frac{x^2}{\xi^3} d\xi  \\
		&\le C\left(1_{x\le 1}+e^{-cx}1_{x>1}\right) \|\psi\|_{L^\infty([0,T])}  \int_0^\infty \xi^{-1}G_1^\prime(\xi) d\xi  \\
		&= C\left(1_{x\le 1}+e^{-cx}1_{x>1}\right) \|\psi\|_{L^\infty([0,T])}  \int_0^\infty e^{-\xi^2} d\xi  \\&
		=  C\left(1_{x\le 1}+e^{-cx}1_{x>1}\right) \|\psi\|_{L^\infty([0,T])}.
        \label{eq:bound_u_ptwise}
	\end{split}
    \end{equation}
	Hence,  for any $ t \in [0,T]$, we have
	\begin{equation}
		\left\|u(t,\cdot)\right\|_{L^\infty(\mathbb R_+)},\left\|u(t,\cdot)\right\|_{L^p(\mathbb R_+)}\le C \, \|\psi\|_{L^\infty([0,T])}, \quad \mbox{for any }  1\le p<\infty.
		\label{eq:Lp_norm_u}
	\end{equation}  In a similar way, taking the spatial derivative of \eqref{eq:sol_heat}, we get the following bound for the spatial derivative of $u$, for any $(t,x)\in [0,T]\times \mathbb R_+$ 
	\begin{equation}\label{eq:bound_derivata_u}
		\begin{split}
			|\partial_x u(t,x)| &\le 2\|\psi\|_{L^\infty([0,T])}  \int_0^t \left|{\partial^2_x G}(t-\tau,x)\right| d\tau\\
			&= C\|\psi\|_{L^\infty([0,T])}  \int_{x/\sqrt{t}}^\infty \frac{\xi^3}{x^3} G_1^{\prime\prime}(\xi) \frac{x^2}{\xi^3} d\xi\\
			&\le C\|\psi\|_{L^\infty}([0,T]) \left(\frac{1}{x}1_{x\le 1}+e^{-cx}1_{x>1}\right) \int_0^\infty G_1^{\prime\prime}(\xi) d\xi\\
			&= C\left(\frac{1}{x}1_{x\le 1}+e^{-cx}1_{x>1}\right) \|\psi\|_{L^\infty([0,T])}.
		\end{split}
	\end{equation}
	
	We can estimate the $W^{\alpha,p}(\mathbb R_+)$ seminorm of $u(t,\cdot)$,  for any $ t \in [0,T]$; indeed, we may write
	
	\begin{align*}
		[u(t,\cdot)]_{W^{\alpha,p}(\mathbb R_+)}^p &= \iint_{|x-y|<1} \frac{|u(t,x)-u(t,y)|^p}{|x-y|^{1+\alpha p}} dxdy +\iint_{|x-y|>1} \frac{|u(t,x)-u(t,y)|^p}{|x-y|^{1+\alpha p}} dxdy\\
		&\le 2\iint_{x<y<x+1} \frac{|u(t,x)-u(t,y)|^p}{|x-y|^{1+\alpha p}} dxdy +2^p\iint_{|x-y|>1} \frac{|u(t,x)|^p}{|x-y|^{1+\alpha p}} dxdy\\
		&\le 2\iint_{x<y<x+1} \frac{|u(t,x)-u(t,y)|^p}{|x-y|^{1+\alpha p}} dxdy +C\|u\|_{L^p(\mathbb R_+)}^p,
	\end{align*}
	where the last line is due to  the  equation \eqref{eq:Lp_norm_u}   and the fact that the function $x^{-1-\alpha p}$ is integrable at infinity.  By using the inequality $e^{-cz}\le Cz^{-1}$ and by \eqref{eq:Lp_norm_u}  and \eqref{eq:bound_derivata_u}, we have that, for any $ t \in [0,T]$,
	\begin{align*}
		[u(t,\cdot)]_{W^{\alpha,p}(\mathbb R_+)}^p &\le C\|\psi\|_{L^\infty([0,T])}^p+2\iint_{x<y<x+1} |x-y|^{-1-\alpha p}\left(\int_{x}^{y} |\partial_z u(t,z)| dz\right)^p dxdy\\
		&\le C\|\psi\|_{L^\infty([0,T])}^p\\&\,\,+2\iint_{x<y<x+1} |x-y|^{-1-\alpha p}\left(\int_{x}^{y} C(z^{-1}1_{z\le 1}+e^{-cz}1_{z>1}) \|\psi\|_{L^\infty([0,T])} dz\right)^p dxdy\\
		&\le C\|\psi\|_{L^\infty([0,T])}^p\left( 1+  \iint_{x<y<x+1} |x-y|^{-1-\alpha p} \left(|\log(y/x)|^p1_{x\le 1} +|x-y|^p e^{-cpx}1_{x>1}\right) dxdy\right)\\
		&\le C\|\psi\|_{L^\infty([0,T])}^p\left( 1+ \int_0^1 dx \int_x^{\infty} |x|^{-1-\alpha p}|1-y/x|^{-1-\alpha p}|\log(y/x)|^p dy\right)\\
		& \,\, +C\|\psi\|_{L^\infty([0,T])}^p \int_1^\infty dx \int_x^{x+1} |x-y|^{-1+(1-\alpha)p} e^{-cpx} dy .
	\end{align*}
	By  the changes of variable $w=y/x$ and $v=y-x$, in the first and second integral, respectively, we obtain, for any $ t \in [0,T]$,
	\begin{align*}
		[u(t,\cdot)]_{W^{\alpha,p}(\mathbb R_+)}^p &\le C\|\psi\|_{L^\infty([0,T])}^p+ C\|\psi\|_{L^\infty([0,T])}^p \int_0^1 dx \,|x|^{-\alpha p}\int_1^\infty |1-w|^{-1-\alpha p}|\log w|^p du   \\
		&\,\,+ C\|\psi\|_{L^\infty([0,T])}^p \int_1^\infty dx \int_0^1 v^{-1+(1-\alpha)p}e^{-cpx}  dv.
	\end{align*}
 
	The second integral in the above expression is clearly finite. Concerning the first integral, note that, for $w$ large, $|1-uw|^{-1-\alpha p}|\log w|^p$ is integrable, while for $w$ close to $1$,
	\begin{align*}
		|1-w|^{-1-\alpha p}|\log w|^p \approx |1-w|^{-\alpha p}.
	\end{align*}
	Hence, the first integral is finite if and only if $\alpha p<1$. Hence, if $0<\alpha<1$ and $1< p<1/\alpha$, we reach the conclusion.
\end{proof}

\medskip

Given more regularity on the boundary condition, we get a bound on the norm of $u$ in a $2+\alpha$ fractional Sobolev space, uniformly in time, with $\alpha \in (0,1)$.

\begin{proposition}\label{prop:bound_u_W_alpha,p_Lipschitz}
	Let $u$ be the solution of the system \eqref{eq:u}-\eqref{eq:u_boundary_condition}, with the boundary condition such that $\psi$ is Lipschitz and $\psi(0)=0$. Then, for any $ 0<\alpha<1$ and $1< p<1/\alpha$, and  for a constant $C$ not depending upon $\psi$, we have   
	\begin{align}
		\sup_{t\in [0,T]}\|u(t,\cdot)\|_{W^{2+\alpha,p}(\mathbb R_+)} \le C\|\psi\|_{W^{1,\infty}([0,T])}.\label{eq:bound_u_2}
	\end{align} 
\end{proposition}
\begin{proof}
	Note that $\partial_x^2G =\partial_t G$,due to \eqref{eq:heat_kernel} and \eqref{eq:partial_2_heat_kernel}; furthermore from \eqref{eq:u_initial_condition} it follows that  $G(0,x)=0$ for $x>0$. By an integration by parts in time, we get, for any $(t,x)\in [0,T]\times \mathbb R_+$ 
	\begin{align*}
		\partial_x u(t,x) &= -2 \int_0^t {\partial^2_x G}(\tau,x)\psi(t-\tau)d\tau\\
		&= -2 \int_0^t {\partial_\tau G}(\tau,x)\psi(t-\tau)d\tau\\
		&= -2G(\tau,x)\psi(t-\tau)\mid_{\tau=0}^t -2\int_0^t G(\tau,x)\dot{\psi}(t-\tau) d\tau\\
		&= -2\int_0^t G(\tau,x)\dot{\psi}(t-\tau) d\tau.
	\end{align*}
	Differentiating \eqref{eq:sol_heat} with respect to $x$, we get
	\begin{equation}\label{eq:heat_kernel_partial_x_u}
		\partial_x^2 u(t,x) = -2\int_0^t \partial_x G(\tau,x)\dot{\psi}(t-\tau) d\tau.
	\end{equation}
	Hence, by comparing  \eqref{eq:heat_kernel} and \eqref{eq:heat_kernel_partial_x_u}, one can deduce that for any   $(t,x)\in [0,T]\times \mathbb R_+$, $\partial_x^2 u(t,x)$ is the solution to the heat equation \eqref{eq:u}, with boundary condition at $x=0$ given by $\dot{\psi}$. Hence, applying the bound \eqref{eq:bound_u_1} to $\partial_x^2 u$, with $\dot{\psi}$ replacing $\psi$, we get
	\begin{align}
		\sup_{t\in [0,T]}\|\partial_x^2  u(t,\cdot)\|_{W^{\alpha,p}(\mathbb R_+)}\le C\|\dot{\psi}\|_{L^\infty([0,T])}.\label{eq:bound_partial_x^2_1}
	\end{align}
	By \eqref{eq:bound_u_1} and  \eqref{eq:bound_partial_x^2_1}  we finally obtain the bound of the norm \eqref{eq:def_norm_sobolev_fractional}  with $k=2$, that is the statement \eqref{eq:bound_u_2}.
	
\end{proof}

The next step is to give the main result of the section, i.e a continuous result for the solution $u$ of the system \eqref{eq:u}-\eqref{eq:u_boundary_condition}   in $W^{1,p}(\mathbb R_+)$, uniformly in time, with respect to the boundary condition. In order to obtain the result we consider compatible Banach spaces in Propositions \ref{prop:bound_u_W_alpha,p} and  \ref{prop:bound_u_W_alpha,p_Lipschitz}. Let us just recall the definition of interpolation pair of Banach spaces. The interest reader may refer to \cite{1988_bennet_sharpley,1964_Lions}. Here, for the interpolation methods and results, we mainly refer to    \cite{2018_Lunardi,1978_triebel}.

\begin{definition}
	$$
	X_0\cap X_1 \subset X_0,X_1 \subset X_0+X_1,
	$$
	where both $
	X_0\cap X_1 ,  X_0+X_1$ are Banach space with the norms
	\begin{eqnarray*}
		\|x\|_{X_0 \cap X_1} &:=& \max \left ( \left \|x \right \|_{X_0}, \left \|x \right \|_{X_1} \right ),\\
		\|x\|_{X_0 + X_1} &:=& \inf \left \{ \left \|x_0 \right \|_{X_0} + \left \|x_1 \right \|_{X_1} \ : \  x = x_0 + x_1, \; x_0 \in X_0, \; x_1 \in X_1 \right \}.
	\end{eqnarray*}
	An \emph{interpolation space} is a space $X$ that is an intermediate space between $X_0$ and $X_1$ in the sense that
	$$
	X_0 \cap X_1 \subset X \subset X_0 + X_1,
	$$
	where the two inclusions maps are continuous. Given two compatible couples $(X_0, X_1)$ and $(Y_0, Y_1)$, an \emph{interpolation pair} is a couple $(X, Y)$ of Banach spaces with the two following properties:
	\begin{itemize}
		\item[i)] The space $X$ is intermediate between $X_0$ and $X_1$, and $Y$ is intermediate between $Y_0$ and $Y_1$.
		\item[ii)]If $L$ is any linear operator from $X_0 + X_1$ to $Y_0 + Y_1$, which maps continuously $X_0$ to $Y_0$ and $X_1$ to $Y_1$, then it also maps continuously $X$ to $Y$.
	\end{itemize}
\end{definition}

\begin{definition}
	Let $(X_0, X_1)$ a compatible couple of Banach spaces , if $0<\theta<1$ and $1\le q\le \infty$	 with  $ (X_0, X_1)_{\theta,q} $ one denotes the Banach space obtained by the real interpolation $K$ method with parameters $\theta$  and $q$. To be more precise,
	for $ t > 0$ and any $x \in X_0 + X_1$, let
	$$ 
	K(x, t; X_0, X_1) = \inf \left \{ \left \|x_0 \right \|_{X_0} + t \left \|x_1 \right \|_{X_1} \ :\  x = x_0 + x_1, \; x_0 \in X_0, \, x_1 \in X_1 \right \}.
	$$
	
	Let
	\begin{align*}
		\|x\|_{\theta,q} &= \left( \int_0^\infty \left( t^{-\theta} K(x, t; X_0, X_1) \right)^q \, \tfrac{dt}{t} \right)^{\frac{1}{q}}=\|K(x, t; X_0, X_1)\|_{L_*^q(\mathbb R_+)}, && 0 < \theta < 1, \, 1 \leq q < \infty, \\
		\|x\|_{\theta,\infty} &= \sup_{t > 0} \; t^{-\theta} K(x, t; X_0, X_1), &&  0 \le \theta \le 1,
	\end{align*}
	where $L_*^q(\mathbb R)$ denotes the $L^q$ space with respect to the measure $dt/t$.
	Then 	$ (X_0, X_1)_{\theta,q} $  is  the following Banach space
	$$
	(X_0, X_1)_{\theta,q} =\{ x\in X_0+X_1: 	\|x\|_{\theta,q} < \infty \}.
	$$
\end{definition}

\medskip

\begin{proposition}\label{prop:heat_bd}
	Assume that $0<\beta<1/2$ and take 
	\begin{equation}\label{eq:parameter_bound_p_beta}
		1\le q<\frac{1}{1-2\beta}.
	\end{equation} Let $u$ be the solution of the system \eqref{eq:u}-\eqref{eq:u_boundary_condition}. Then, for every $\psi \in C^\beta[0,T]$ with $\psi(0)=0$, and  for a constant $C$ not depending upon $\psi$, we have  
	\begin{align}
		\|u\|_{L^\infty([0,T],W^{1,q}(\mathbb R_+))} =	\sup_{t\in [0,T]}\|u(t,\cdot)\|_{W^{1,q}(\mathbb R_+)} \le C\|\psi\|_{C^\beta([0,T])}.\label{eq:u_bound}
	\end{align} 
\end{proposition}

\begin{proof} 
	Given the two compatible couples $\left(C([0,T]),C^1([0,T])\right) $  and  $\left( W^{\alpha,p}(\mathbb R_+),W^{2+\alpha,p}(\mathbb R_+)\right)$ we seek for an interpolation couple where the first component is the space to which the initial condition belongs and the second gives the regularity of the solution of the system \eqref{eq:u}-\eqref{eq:u_boundary_condition}.
	
	\medskip
	
	Let us denote by  $A$ the solution operator given, where defined, by \eqref{eq:sol_heat}, that is
	\[
	A\psi=u.
	\]     
	
	By \eqref{eq:bound_u_1} and \eqref{eq:bound_u_2},  $A$  is a bounded linear operator from   $C([0,T]) +C^1([0,T]) $ to $W^{\alpha,p}(\mathbb R_+)+W^{2+\alpha,p}(\mathbb R_+)$, which maps continuously     $C([0,T])  $ to $W^{\alpha,p}(\mathbb R_+) $ and  $C^1([0,T])$ to $W^{2+\alpha,p}(\mathbb R_+)$.  To be precise, $A$ is only defined on subspaces of $C([0,T])$ and $C^1([0,T])$ of functions $\psi$ such that $\psi(0)=0$. However, we can easily extend $A$ to the full spaces by composing it with the operator $\psi\mapsto \psi-\psi(0)$. 
	
	\medskip
	
	By  Theorem 1.1.6 and Proposition 1.1.4 in \cite{2018_Lunardi}, for any $0<\beta'<1$ and  $\epsilon>0$ small, $A$ is a bounded operator such that
	\begin{equation}
		(C([0,T]),C^1([0,T]))_{\beta'+\epsilon,\infty} \hookrightarrow (C([0,T]),C^1([0,T]))_{\beta',p}\stackrel{A}{\longrightarrow} (W^{\alpha,p}(\mathbb R_+),W^{2+\alpha,p}(\mathbb R_+))_{\beta',p}.
		\label{eq:interpolation_operator}
	\end{equation}
	By \cite{1978_triebel} (Section 4.4.1) and \cite{2018_Lunardi} (Examples 1.1.8 and 1.3.7), for any  $0<\beta'<1$ and  
  $\epsilon >0 $ small the interpolation space at the left side of the map \eqref{eq:interpolation_operator} is
		\begin{equation*}
		(C([0,T]),C^1([0,T]))_{\beta'+\epsilon,\infty}= C^{\beta'+\epsilon}([0,T]). 
	\end{equation*}
	By \cite{1978_triebel}(Sections, 4.3.1. and  4.5.2.) and \cite{2018_Lunardi} (Examples 1.3.8 and 1.3.10.), whenever  $2\beta'+\alpha$ is not an integer,  the  interpolation space at the right hand side of the map \eqref{eq:interpolation_operator} is
	\begin{equation}
		(W^{\alpha,p}(\mathbb R_+), W^{2+\alpha,p}(\mathbb R_+))_{\beta',p}=W^{(1-\beta')\alpha+\beta'(2+\alpha),p}(\mathbb R_+)=W^{2\beta'+\alpha,p}(\mathbb R_+).
		\label{eq:interpolation_W} 
	\end{equation}
	Hence, by \eqref{eq:interpolation_operator}-\eqref{eq:interpolation_W} we obtain, for $0<\alpha<1$, $1<p<1/\alpha$, $0<\beta'<1$ with $2\beta'+\alpha\notin \mathbb{N}$,
	\begin{align}\label{eq:u_psi_2beta_alpha_beta_epsilon}
		\sup_{t\in [0,T]}\|u(t,\cdot)\|_{W^{2\beta+\alpha,p}(\mathbb R_+)} \le C\|\psi\|_{C^{\beta+\epsilon}([0,T])}.
	\end{align}
	 For $0<\beta'<1/2$, by the Sobolev embeddings  \cite{2018_Lunardi,1978_triebel}, we may apply   Proposition \ref{prop:Sobolev_embedding}. 
	We consider  $s=2\beta'+\alpha \in (1,2)$  and $s>1$, $s^\prime =1$, and $q$ such that 
	\begin{align}
		\frac{1}{q} = \frac{1}{p}+1-(2\beta'+\alpha) > 1-2\beta',\label{eq:exponent_condition}
	\end{align}
	so that $1\le q<1/(1-2\beta')$.  Hence, condition \eqref{eq:prop_embedding_parameters} is satisfied and from \eqref{eq:prop_embedding}, we get
	\begin{equation}\label{eq:u_psi_2beta_alpha_in_q}
		W^{2\beta'+\alpha,p}(\mathbb R_+)\subseteq W^{1,q}(\mathbb R_+).
	\end{equation}
	In conclusion, from \eqref{eq:u_psi_2beta_alpha_beta_epsilon} and \eqref{eq:u_psi_2beta_alpha_in_q} we can state that for given $0<\beta'<1/2$, and  $1\le q<1/(1-2\beta')$, one can always find $0<\alpha<1$ and $1<p<1/\alpha$ satisfying \eqref{eq:exponent_condition}, such that  
	\begin{align*}
		\sup_{t\in [0,T]}\|u(t,\cdot)\|_{W^{1,q}(\mathbb R_+)} \le C\|\psi\|_{C^{\beta'+\epsilon}(0,T)}.
	\end{align*}
	 Taking $\beta'=\beta-\epsilon$ for $0<\epsilon<\beta$ such that $2\beta'+\alpha\in (1,2)$ and $q<1/(1-2\beta')$, we get \eqref{eq:u_bound}.
\end{proof}

\section{A linear parabolic equation: a priori estimates }\label{sec:uncoupled_equation}

In this section we  study  an auxiliary problem. Precisely, we prove the well posedness of the  following linear PDE 
\begin{align}
	\begin{aligned}\label{eq:s_lin}
		&\partial_t \tilde{s} = \partial_x^2 \tilde{s} +b(t,x) \partial_x \tilde{s} +\widetilde{\gamma}(t,x) \tilde{s},\\
		&\tilde{s}(0,x)=s_0,\\
		&\tilde{s}(t,0)=\psi(t),
	\end{aligned}
\end{align}
 with coefficients  independent on $\tilde{s}$ and satisfying the following assumptions.
  
\begin{assumption}\label{hp:linearPDE}
Let us assume the following hypothesis upon the functions in \eqref{eq:v_lin} and the boundary condition $\psi$.
	\begin{itemize}
		\item[A1.] The boundary condition $\psi\in C^\beta([0,T])$,  for some $\beta\in \left(  1/4,1/2\right)$ and it is such that $\psi(0)=0$.
		\item[A2.] The functions $b$ and $\widetilde\gamma$ satisfy the following regularity conditions
		\begin{align*}
	b\in L^\infty([0,T], L^2(\mathbb R_+)), \quad \widetilde{\gamma}\in L^\infty([0,T] \times \mathbb R_+).	
		\end{align*}
		\item[A3.] The initial condition $v_0 \in W^{1,p}(\mathbb R_+)$ for some $2\le p\le \infty$ and it is such that $v_0(0)=0$.
	\end{itemize}
\end{assumption}

The result of this Section are used in the next Section to deal with non-smooth (non $C^{1,2}$) solutions to a linearized version of the system \eqref{eq:det_nostro_s}-\eqref{eq:det_nostro_c}.

\medskip

The solution $\tilde{s}:[0,T]\times[0,\infty)\to \R$  of system \eqref{eq:s_lin}  is meant as a mild solution \cite{1984_cannon,2015_liu_Rochner} according with the following definition. 
  \begin{definition}[Mild solution for $\tilde{s}$] 
  A function $\tilde{s}:[0,T]\times[0,\infty)\to \R$ is a $L^\infty ([0,T],W^{1,p}(\mathbb R_+))$ \emph{mild solution} of the equation \eqref{eq:s_lin} if   $\tilde{s} \in L^\infty ([0,T],W^{1,p}(\mathbb R_+))$ and it is such that, for every $t\in [0,T]$, 
  	\begin{align}
  		\tilde{s}(t,\cdot) = -2\int^t_0\partial_x  {G}(t-\tau,\cdot) \psi(\tau)d\tau +  \overline{G}(t,\cdot)*_Ds_0 +\int_0^t  \overline{G}(t,\cdot)*_D\left(b(\tau,\cdot)\partial_x \tilde{s}(\tau,\cdot)+\widetilde{\gamma}(\tau,\cdot) \tilde{s}(\tau,\cdot)\right) d\tau,\label{eq:PDE_s_lin_mild}
  	\end{align}
  	where $G$ is the heat kernel \eqref{eq:heat_kernel}  and $ \overline{G}(t,\cdot)*_Df (x)$  denotes the convolution  \eqref{eq:K_conv}.	
  \end{definition}
  
  \medskip 
\begin{remark}  Let us comment that  although $ \overline{G}*_Df$  in \eqref{eq:K_conv} is not the convolution with $f$ in the strict sense,  the Young inequality still holds for $G*_Df$, i.e. for $ 1\leq p^\prime, q^\prime, r^\prime\leq \infty$ with
	\begin{equation}\label{eq:parameters_Young_inequality}
		 \displaystyle \frac {1}{p^\prime} = \frac {1}{r^\prime}+\frac {1}{q^\prime} -1,
	\end{equation}
	  for any $f \in 	{\displaystyle L^{q^\prime}(\mathbb {R}_+ )} $, since by \eqref{eq:heat_bd}  $G \in	 	{\displaystyle L^{p^\prime}(\mathbb R_+)}$, then
\begin{equation} 	 \label{eq:young_inequality_K*f} 
	 \| \overline{G}*_Df\|_{L^{p^\prime}}\leq \|G\|_{L^{r^\prime}}\|f\|_{L^{q^\prime}}.
	\end{equation}
    The interest reader may refer to \cite{1984_cannon} for more details. Concerning derivatives, we have that for  $f \in L^r(\mathbb R_+)$,  $1\le r\le \infty$, $ 
		\partial_x (G(t,\cdot)*_Df) = (\partial_x G(t,\cdot))*_Df.
$ Furthermore, whenever $f \in W^{1,r}(\mathbb R_+)$, with $1\le r\le \infty$ and $f(0)=0$, we have also
	\begin{align}\label{eq:young_inequality_partial_K*f} 
		\partial_x ( \overline{G}(t,\cdot)*_Df) = \int_{\mathbb{R}} G(t,\cdot-y)(\partial_x f)^{even}(y) dy:=G((t,\cdot)*(\partial_x f)^{even}(y),
	\end{align}
	where $(\partial_x f)^{even}(x) =1_{x>0}\partial_x f(x) +1_{x<0}\partial_x f(-x)$ is the even extension of $\partial_x f$.
	
\end{remark}

\medskip

Now let us go back to the problem \eqref{eq:s_lin}.

\medskip

We split the solution $\tilde{s}$ as $\tilde{s}=\tilde{u}+\tilde{v}$ with $\tilde{u}$ again is the solution of the initial boundary problem for the homogeneous heat equation \eqref{eq:u}-\eqref{eq:u_boundary_condition} and $\tilde{v}$ solution of the following non-homogeneous problem
\begin{align}
	\begin{aligned}\label{eq:v_lin}
		&\partial_t \tilde{v} = \partial_x^2 \tilde{v} +b\partial_x \tilde{v} +\widetilde\gamma \tilde{v} +b\partial_x \tilde{u} +\widetilde{\gamma} \tilde{u}, \qquad  &(t,x)\in (0,T]\times \mathbb R_+;\\
		&\tilde{v}(t,0)=0, &\qquad  t\in (0,T];\\
		&\tilde{v}(0,x)=v_0(x)=s_0(x), &\qquad  x\in \mathbb R_+.\\
	\end{aligned}
\end{align}

\medskip

\begin{definition}[Mild solution for $\tilde{v}$]
Let $2\le p\le \infty$. A function $\tilde{v}:[0,T]\times[0,\infty)\to \R$ is a $L^\infty ([0,T],W^{1,p}(\mathbb R_+))$\emph{ mild solution} of  \eqref{eq:v_lin} if it belongs to $L^\infty ([0,T],W^{1,p}(\mathbb R_+))$ and satisfies, for any $t\in [0,T]$,
	\begin{align}
		\begin{split}
		\tilde{v}(t,\cdot) &=  \overline{G}(t,\cdot)*_Dv_0  +
        \int_0^t  \overline{G}(t-\tau, \cdot)*_D(b(\tau,\cdot) \partial_x \tilde{v}(\tau,\cdot) +\widetilde\gamma(\tau,\cdot) \tilde{v}(\tau,\cdot)) d\tau \\
		&\hspace{1.5cm}+\int_0^t  \overline{G}(t-\tau, \cdot)*_D\left(b(\tau,\cdot)\partial_x \tilde{u}(\tau,\cdot)+ \widetilde\gamma(\tau,\cdot) \tilde{u}(\tau,\cdot)\right) d\tau,
		\end{split}\label{eq:PDE_lin_mild}
	\end{align} where  $G$ is the heat kernel \eqref{eq:heat_kernel}  and the operator $ \overline{G}(t,\cdot)*_Df $ is defined by \eqref {eq:K_conv}.
\end{definition}

\medskip

Our aim  is to show that under the regularity conditions upon the coefficient functions,  the boundary function and initial condition,  stated in Assumption \ref{hp:linearPDE},   the equation \eqref{eq:v_lin} is well-posed in $L^\infty ([0,T],W^{1,p}(\mathbb R_+))$ as a mild solution.
First we provide some a priori bounds for $\tilde{v}$.
\begin{proposition}
	Under Assumption \ref{hp:linearPDE}, for $ p\ge q \ge 2$  such that condition \eqref{eq:parameter_bound_p_beta} is fulfilled and  $r>0$ such that
	\begin{equation}\label{eq:PDE_lin_bounds_v_parameters}
	  \frac{1}{r}= \frac{1}{2} + \frac{1}{p}- \frac{1}{q},
	\end{equation} for the function $\tilde{v}$ mild solution to \eqref{eq:PDE_lin_mild} the following bounds hold, for any $t\in [0,T]$
	\begin{equation}\label{eq:v_apriori_1} 
		\begin{split}			 
				\|\tilde{v}(t,\cdot)\|_{L^p(\mathbb R_+)} \le   \|v_0\|_{L^p(\mathbb R_+)} & + CT^{3/4} \left(\|b\|_{L^\infty([0,T],L^2(\mathbb R_+))}+\|\gamma\|_{L^\infty([0,T],L^\infty(\mathbb R_+))}\right)\,\|\tilde{v}\|_{L^\infty([0,T],W^{1,p}(\mathbb R_+))}\\
				&   +T^{(1+1/r)/2}\left(\|b\|_{L^\infty([0,T],L^2(\mathbb R_+))} +\|\gamma\|_{L^\infty_t(L^\infty_x)}\right)\|\psi\|_{C^\beta([0,T])}
		\end{split}
	\end{equation}
	\begin{equation}\label{eq:v_apriori_2} 
	\begin{split}			 
 	\|\partial_x \tilde{v}(t,\cdot)\|_{L^p(\mathbb R_+)} \le  \|\partial_x v_0\|_{L^p(\mathbb R_+)} &+CT^{1/4} \left(\|b\|_{L^\infty([0,T],L^p(\mathbb R_+))}+\|\gamma\|_{L^\infty_t(L^\infty_x)}\right)\|\tilde{v}\|_{L^\infty([0,T],W^{1,p}(\mathbb R_+))}\\ & +CT^{1/(2r)}\left(\|b\|_{L^\infty([0,T],L^2(\mathbb R_+))} +\|\gamma\|_{L^\infty([0,T],L^\infty(\mathbb R_+))}\right) \|\psi\|_{C^\beta([0,T])}
	\end{split}
\end{equation}
	
\end{proposition}
\begin{proof}

In the following,  for any $a,b\ge 2$, we denote by $h(a,b)$ the positive parameter such that $1/h(a,b)=1/a+1/b$; hence
$$
\frac{1}{b} = \frac{1}{h(a,b)}+\frac{1}{a/(a-1)}-1, 
$$ so that \eqref{eq:parameters_Young_inequality}  is satisfied. In particular, 
  from  the previous relation and by hypothesis \eqref{eq:PDE_lin_bounds_v_parameters} upon the parameters, one have with $(a=2,b=p)$ 
\begin{equation*}
\begin{split}
	\frac{1}{p} &= \frac{1}{h(2,p)}+\frac{1}{2}-1, \\ 
 \frac{1}{p} &= \frac{1}{q}+	\frac{1}{r}- \frac{1}{2} =    \frac{1}{h(2,q)}+	\frac{1}{r} - 1. 
\end{split}
\end{equation*}
Hence, for the set of parameters $(p,h(2,p),2)$ and $(p,r,h(2,q))$ condition \eqref{eq:PDE_lin_bounds_v_parameters} is satisfied and we may apply the Young inequality \eqref{eq:young_inequality_K*f}.  For simplicity we introduce here the following notations for the involved Banach  spaces $L^p_x=L^p(\mathbb R_+)$ in space and $L_t^p(L_x^q)=L^p([0,T],L^q(\mathbb R_+))$ in time and space.
Note that  by   Sobolev embedding (e.g.  \cite{1978_triebel} Eq. 4.6.1-(e)]) \begin{equation}\label{eq:embedding_Linfty_W_1,q}\|\tilde{u}\|_{L^\infty_x}\le C\|\tilde{u}\|_{W^{1,q}_x},\end{equation} and by H\"older inequality (since $q\le p$) $$\|\tilde{u}\|_{L^p_x}\le \|\tilde{u}\|_{L^q_x}^{q/p}\|\tilde{u}\|_{L^\infty_x}^{1-q/p},$$ we obtain \begin{equation}\label{eq:W^{1,p}_in_L^p}\|\tilde{u}\|_{L^p_x}\le C\|\tilde{u}\|_{W^{1,q}_x}.
\end{equation}

\medskip

	In order to estimate  the $L^p(\mathbb R_+)$ norm of $\tilde{v}$, we apply recursively  Young inequality for functions in $L^p$ spaces    and their convolution with the kernel $G$	
\eqref{eq:young_inequality_K*f} and  H\"older inequality (with $p,q\ge 2$). Hence, by considering also the $L^p$ norm of $G$ in \eqref{eq:heat_bd} and \eqref{eq:W^{1,p}_in_L^p}, we get from \eqref{eq:PDE_lin_mild}

	\begin{align*}
		\|\tilde{v}(t,\cdot)\|_{L^p_x} &\le \| \overline{G}(t,\cdot)*_D v_0\|_{L^p_x} +\int_0^t \left[\| \overline{G}(t-\tau,\cdot)*_D \left(b(\tau,\cdot)\partial_x \tilde{v}(\tau,\cdot)\right)\|_{L^p_x}+\| \overline{G}(t,\cdot)*_D(\widetilde\gamma(\tau,\cdot) \tilde{v}(\tau,\cdot))\|_{L^p_x}\right.\\
		& \hspace{3.5cm} +\left.\| \overline{G}(t,\cdot)*_D(b(\tau,\cdot)\partial_x \tilde{u}(\tau,\cdot))\|_{L^p_x}+\| \overline{G}(t,\cdot)*_D(\widetilde{\gamma}(\tau,\cdot) \tilde{u}(\tau,\cdot))\|_{L^p_x}\right] d\tau \\
		&\le \|G(t,\cdot)\|_{L^1_x}\|v_0\|_{L^p_x} +\int_0^t \left[
		\|G(t,\cdot)\|_{L^2_x}\|b(\tau,\cdot) \partial_x \tilde{v}(\tau,\cdot)\|_{L^{h(2,p)}_x}
		+\|G(t,\cdot)\|_{L^1_x}\|\widetilde\gamma(\tau,\cdot) \tilde{v}(\tau,\cdot)\|_{L^p_x}\right.\\		
		& \hspace{3.5cm} \left.+\|G(t,\cdot)\|_{L^r_x}\|b(\tau,\cdot)\partial_x \tilde{u}(\tau,\cdot)\|_{L^{h(2,q)}_x}
		+\|G(t,\cdot)\|_{L^1_x}\|\widetilde\gamma(\tau,\cdot) \tilde{u}(\tau,\cdot)\|_{L^p_x} \right] d\tau \\
	 	&\le c_1 \|v_0\|_{L^p_x} +C\int_0^t \left[
		(t-\tau)^{-1/4}\|b(\tau,\cdot)\|_{L^2_x}\|\partial_x \tilde{v}(\tau,\cdot)\|_{L^p_x}
		+\|\widetilde\gamma(\tau,\cdot)\|_{L^\infty_x} \|\tilde{v}(\tau,\cdot)\|_{L^p_x}\right.
		\\
		&\hspace{3cm}\left.+(t-\tau)^{-(1-1/r)/2}\|b(\tau,\cdot)\|_{L^2_x}\|\partial_x \tilde{u}(\tau,\cdot)\|_{L^q_x} +\|\widetilde\gamma_s\|_{L^\infty_x} \|\tilde{u}(\tau,\cdot)\|_{L^p_x} \right] d\tau \\
		&\le c_1\|v_0\|_{L^p_x} +(CT^{3/4}\|b\|_{L^\infty_t(L^2_x)}+CT\|\widetilde\gamma\|_{L^\infty_t(L^\infty_x)}) \|\tilde{v}(\tau,\cdot)\|_{L^\infty_t(W^{1,p}_x)}\\
		& \hspace{2cm} +(CT^{(1+1/r)/2}\|b\|_{L^\infty_t(L^2_x)}+CT\|\widetilde\gamma\|_{L^\infty_t(L^\infty_x)}) \|\tilde{u}(\tau,\cdot)\|_{L^\infty_t(W^{1,q}_x)}  \\
	\end{align*}
	
	The conclusion \eqref{eq:v_apriori_1} follows from the last inequality and  Proposition \ref{prop:heat_bd}, because of Assumption \ref{hp:linearPDE}-A1.

	\bigskip
	
	For obtaining the estimate for the $L^p_x$ norm of $\partial_x \tilde{v}$, we write the equation for $\partial_x \tilde{v}$, using the expressions for $\partial_x (G(t,\cdot)*_Df)$ given in \eqref{eq:young_inequality_partial_K*f} and the fact that $v_0(0)=0$:
	\begin{align*}
		\partial_x \tilde{v}(t,\cdot) &= G(t,\cdot)*_D(\partial_x v_0)^{even} +\int_0^t \partial_x G(t,\cdot)*_D\left(b(\tau,\cdot)\partial_x \tilde{v}(\tau,\cdot)+\widetilde\gamma(\tau,\cdot) \tilde{v}(\tau,\cdot)\right) d\tau\\
		& \hspace{3cm} +\int_0^t \partial_x G(t,\cdot)*_D\left(b(\tau,\cdot)\partial_x \tilde{u}(\tau,\cdot)+\widetilde\gamma(\tau,\cdot) \tilde{u}(\tau,\cdot)\right) d\tau.
	\end{align*}
In a similar way as for inequality \eqref{eq:v_apriori_1}, we prove the bound \eqref{eq:v_apriori_2}. Hence, by \eqref{eq:parameters_Young_inequality}-\eqref{eq:young_inequality_partial_K*f} ,\eqref{eq:PDE_lin_bounds_v_parameters} and \eqref{eq:W^{1,p}_in_L^p}, we obtain
	\begin{align*}
		\|\partial_x \tilde{v}(t,\cdot)\|_{L^p_x} &\le \|G(t,\cdot)*_D(\partial_x v_0)^{even}\|_{L^p_x}
		+\int_0^t \left[\|\partial_x G(t,\cdot)*_D(b(\tau,\cdot)\partial_x \tilde{v}(\tau,\cdot))\|_{L^p_x}+\|\partial_x G(t,\cdot)*_D(\widetilde\gamma(\tau,\cdot) \tilde{v}(\tau,\cdot))\|_{L^p_x} \right.
		\\
		&\hspace{4cm}+\left. \|\partial_x G(t,\cdot)*_D(b(\tau,\cdot) \partial_x \tilde{u}(\tau,\cdot))\|_{L^p_x}+ \|\partial_x G(t,\cdot)*_D(\widetilde\gamma(\tau,\cdot)  \tilde{u}(\tau,\cdot))\|_{L^p_x}\right] d\tau\\
		&\le \|G(t,\cdot)\|_{L^1_x}\|\partial_x v_0\|_{L^p_x} +\int_0^t \left[
		\|\partial_x G(t,\cdot)\|_{L^2_x}\|b(\tau,\cdot) \partial_x \tilde{v}(\tau,\cdot)\|_{L^{h(2,p)}_x}
		+\|\partial_x G(t,\cdot)\|_{L^1_x}\|\widetilde\gamma(\tau,\cdot) \tilde{v}(\tau,\cdot)\|_{L^p_x}\right.\\
		&\hspace{4cm} \left.
		+\|\partial_x G(t,\cdot)\|_{L^r_x}\|b(\tau,\cdot)\partial_x \tilde{u}(\tau,\cdot)\|_{L^{h(2,q)}_x}
		+\|\partial_x G(t,\cdot)\|_{L^1_x}\|\widetilde\gamma(\tau,\cdot) \tilde{u}(\tau,\cdot)\|_{L^p_x} \right] d\tau\nonumber\\
		&\le c'_1 \|\partial_x v_0\|_{L^p_x} +C\int_0^t \left[
		(t-\tau)^{-3/4} \|b(\tau,\cdot)\|_{L^2_x}\|\partial_x \tilde{v}(\tau,\cdot)\|_{L^p_x}
		+(t-\tau)^{-1/2}\|\widetilde\gamma(\tau,\cdot)\|_{L^\infty_x} \|\tilde{v}(\tau,\cdot)\|_{L^p_x}\right.\\
		&\hspace{3cm}\left.	+(t-\tau)^{-(2-1/r)/2}\|b(\tau,\cdot)\|_{L^2_x}\|\partial_x \tilde{u}(\tau,\cdot)\|_{L^q_x} +(t-\tau)^{-1/2}\|\widetilde\gamma_s\|_{L^\infty_x} \|\tilde{u}(\tau,\cdot)\|_{L^p_x} \right] d\tau\\
		&\le c'_1\|\partial_x v_0\|_{L^p_x} +\left(C T^{1/4}\|b\|_{L^\infty_t(L^2_x)}+CT^{1/2} \|\widetilde\gamma\|_{L^\infty_t(L^\infty_x)}\right) \|\tilde{v}(\tau,\cdot)\|_{L^\infty_t(W^{1,p}_x)}\\
		&\hspace{2cm} +\left(CT^{1/(2r)}\|b\|_{L^\infty_t(L^2_x)}+CT^{1/2}\|\widetilde\gamma\|_{L^\infty_t(L^\infty_x)}\right)\|\tilde{u}(\tau,\cdot)\|_{L^\infty_t(W^{1,q}_x)} \\
		\end{align*}
By Proposition \ref{prop:heat_bd} and Assumption \ref{hp:linearPDE}-A1, we obtain the bound \eqref{eq:v_apriori_2}.
	
\end{proof}

Now we are ready to give a result of  existence of  a  mild solution $\tilde{v}$ in $L^\infty([0,T],W^{1,p}(\mathbb R_+))$ that is it satisfies equation   \eqref{eq:PDE_lin_mild}.
\begin{proposition}\label{prop:PDE_lin}
	Under Assumption \ref{hp:linearPDE}, with $2\le p<1/(1-2\beta)$, the PDE \eqref{eq:v_lin} admits a unique   mild solution $\tilde{v}$ in $L^\infty ([0,T],W^{1,p}(\mathbb R_+))$. Furthermore, for some constant $C_T>0$, the solution $\tilde{v}$  satisfies  the following bound
	\begin{equation}\label{eq:bound_v}
		\begin{split}
		\sup_{t\in[0,T]}\|\tilde{v}(t,\cdot)\|_{W^{1,p}(\mathbb R_+)} \le &C_T (\|b\|_{L^\infty_t(L^2_x)}+\|\widetilde\gamma\|_{L^\infty_{t,x}})^4\|v_0\|_{W^{1,p}(\mathbb R_+)} \\ &+C_T \left(\|b\|_{L^\infty([0,T]L^2(\mathbb R_+))}+\|\widetilde\gamma\|_{L^\infty([0,T]\times (\mathbb R_+))}\right)\|\psi\|_{C^\beta}. 
		\end{split}
	\end{equation}
\end{proposition}
\begin{proof} 
	To show the existence and uniqueness of $\tilde{v}$ satisfying \eqref{eq:PDE_lin_mild}, we define a map $F:L^\infty([0,T],W^{1,p}(\mathbb R_+))\to L^\infty([0,T],W^{1,p}(\mathbb R_+))$ such that, for any $h\in L^\infty([0,T],W^{1,p}(\mathbb R_+))$,
	\begin{align*}
		F(h) =  \overline{G}(t,\cdot)*_Dv_0 &+\int_0^t  \overline{G}(t,\cdot)*_D(b(\tau,\cdot)\partial_x h(\tau,\cdot)+\widetilde\gamma(\tau,\cdot) h(\tau,\cdot)) d\tau \\&  +\int_0^t  \overline{G}(t,\cdot)*_D(b(\tau,\cdot)\partial_x \tilde{u}(\tau,\cdot)+\widetilde\gamma(\tau,\cdot) \tilde{u}(\tau,\cdot) d\tau.
	\end{align*}
	We note that $\tilde{v}\in L^\infty([0,T],W^{1,p}(\mathbb R_+))$ satisfies \eqref{eq:PDE_lin_mild} if and only if $\tilde{v}$ is a fixed point for $F$. By proceeding as for the bounds \eqref{eq:v_apriori_1}-\eqref{eq:v_apriori_2}, we obtain that $F$ is well-defined and, for $T$ sufficiently small, such that $$CT^{3/4}+CT^{1/4}\le 1/2,$$ $F$ is a contraction on the time interval $[0,T]$. Therefore, by the fixed point theorem, for small $T>0$, there exists a unique fixed point $\tilde{v}$ of the map $F$ and hence a unique mild solution $\tilde{v}$ in $L^\infty([0,T],W^{1,p}(\mathbb R_+))$. The existence and uniqueness for every $T>0$ follows by a standard iteration argument.
	
	 Finally, the bound \eqref{eq:bound_v} follows by a Gronwall-type inequality \cite{2019_webb}(Theorem 3.2.) applied to the bounds \eqref{eq:v_apriori_1} on $\|\tilde{v}_t\|_{L^p(\mathbb R_+)}$ and \eqref{eq:v_apriori_2} on $\|\partial_x \tilde{v}_t\|_{L^p(\mathbb R_+)}$, together with Proposition \ref{prop:heat_bd}.
\end{proof}

\begin{corollary}\label{cor:wellposed_s}
	Under Assumption \ref{hp:linearPDE}, with $2\le p<1/(1-2\beta)$, the linear PDE \eqref{eq:s_lin} admits a unique $L^\infty([0,T],W^{1,p}(\mathbb R_+))$   mild solution $\tilde{s}$, i.e.    satisfying \eqref{eq:PDE_s_lin_mild}. This solution can be decomposed as $\tilde{s}=\tilde{u}+\tilde{v}$, where $\tilde{u}$ is the unique $L^\infty([0,T],W^{1,p}(\mathbb R_+))$ mild solution  of the system \eqref{eq:u}-\eqref{eq:u_boundary_condition} and $\tilde{v}$ is the unique mild solution in $L^\infty([0,T],W^{1,p}(\mathbb R_+))$ of \eqref{eq:v_lin}.
\end{corollary}

\begin{proof}
	If $\tilde{u}$ is the unique $L^\infty([0,T],W^{1,p}(\mathbb R_+))$ mild solution  of the system \eqref{eq:u}-\eqref{eq:u_boundary_condition},  then from Proposition \ref{prop:PDE_lin} $\tilde{v}$ is a $L^\infty([0,T],W^{1,p}(\mathbb R_+))$ mild solution  of \eqref{eq:PDE_lin_mild} if and only if $\tilde{s}=\tilde{u}+\tilde{v}$ is a $L^\infty([0,T],W^{1,p}(\mathbb R_+))$ mild solution of  \eqref{eq:PDE_s_lin_mild}. Hence, well-posedness for $\tilde{s}$ follows from well-posedness of $\tilde{v}$.
\end{proof}

\begin{remark}\label{rmk:wellposed_lin_infty}
	If $\psi \in C^1_t$, then well-posedness for mild solutions $\tilde{s}$ to \eqref{eq:s_lin} holds also for $p=\infty$,  that is it is a $L^\infty([0,T],W^{1,\infty}(\mathbb R_+))$, with initial condition $v_0$ in $W^{1,\infty}(\mathbb R_+)$: indeed well-posedness in $L^\infty([0,T],W^{1,\infty}(\mathbb R_+))$ holds not only for $\tilde{v}$ solution to \eqref{eq:v_lin} but also for $\tilde{u}$ solution to the heat equation \eqref{eq:u}-\eqref{eq:u_boundary_condition}, by Proposition \ref{prop:bound_u_W_alpha,p}.
\end{remark}

\begin{remark}\label{rmk:time_continuity_lin}
	Assume $2\le p<1/(1-2\beta)$. The mild solution $v$ to \eqref{eq:v_lin} satisfies, for every $s<t$,
	\begin{align*}
		\tilde{v}_t-\tilde{v}_s = \int_0^s  (\overline{G}(t-r, \cdot)- \overline{G}({s-r},\cdot))*_DA_r dr +\int_s^t  \overline{G}({t-r},\cdot)*_DA_rdr
	\end{align*}
	where $A_r=b_r\partial_x \tilde{v}_r+\widetilde\gamma_r\tilde{v}_r +b_r\partial_x \tilde{u}_r+\gamma_r\tilde{u}_r$. Now we take the limit for $t-s\to 0$. For a.e. $r<s$, the term $(\overline{G}(t-r, \cdot)- \overline{G}({s-r},\cdot))*_D A_r$ tends to $0$ in $L^p$ and, proceeding as in \eqref{eq:v_apriori_1}, one can show that the $L^p$ norm of $(\overline{G}(t-r, \cdot)- \overline{G}({s-r},\cdot))*_D A_r$ is bounded by a time-integrable function, uniformly in $s$ and $t$, hence $\int_0^s (\overline{G}(t-r, \cdot)- \overline{G}({s-r},\cdot))*_D A_r dr$ tends to $0$ in $L^p$. Proceeding again as in \eqref{eq:v_apriori_1}, one can also show that the $\int_s^t \overline{G}({t-r},\cdot)*_D A_rdr$ tends to $0$ in $L^p$. Hence $t\mapsto v_t$ is actually continuous with values in $L^p$. 	
	With a similar argument, one gets that $t\mapsto \tilde{u}_t$ is continuous in $L^p$. Hence also $t\mapsto s_t$ is continuous in $L^p$, for $2\le p<1/(1-2\beta)$.
\end{remark}

We conclude the study of the linearized system solution with some  stability bounds for both the components of the splitting, which will be useful later.

\begin{proposition}\label{prop:stability_lin}
	Let   $b^i$, $\widetilde\gamma^i$, $v^i_0$, $\psi^i$, $i=1,2$ satisfy Assumption \ref{hp:linearPDE} with $p=2$ and let us denote by   $\tilde{u}^i$ and  $\tilde{v}^i$, $i=1,2$, the corresponding $L^\infty([0,T],W^{1,2}(\mathbb R_+))$ solutions of  \eqref{eq:u}-\eqref{eq:u_boundary_condition} and \eqref{eq:PDE_lin_mild}, respectively. Let us assume in addition that $\widetilde\gamma^1$, $\widetilde\gamma^2$ are in $C([0,T],L^2(\mathbb R_+))$. Then we have
	\begin{align*}
		\sup_{t\in[0,T]}\|\tilde{u}^1-\tilde{u}^2\|_{W^{1,2}(\mathbb R_+)} &\le C\|\psi^1-\psi^2\|_{C^\beta_t},\\
		\sup_{t\in[0,T]}\|\tilde{v}^1-\tilde{v}^2\|_{W^{1,2}(\mathbb R_+)} &\le \rho_{1,b,\tilde{\gamma}}  \|v^1_0-v^2_0\|_{W^{1,2}(\mathbb R_+)}+\rho_{2,b,\tilde{\gamma}}\|\psi^1-\psi^2\|_{C^\beta([0,T])}\\
		&\quad +\rho_{3,b,\tilde{\gamma},v_0,\psi} \left(\|b^1-b^2\|_{L^\infty([0,T],L^2(\mathbb R_+))}+\|\widetilde\gamma^1-\widetilde\gamma^2\|_{C([0,T],L^2(\mathbb R_+))}\right)
	\end{align*}
	where $\rho_{\cdot}$ are Borel locally bounded function such that $\rho_{i,b,\tilde{\gamma}}=\rho_i\left(\|b^*\|_{L^\infty([0,T],L^2(\mathbb R_+))},\|\widetilde\gamma^*\|_{L^\infty([0,T]\times\mathbb R_+)}\right)$, for $i=1,2$ and $\rho_{3,b,\tilde{\gamma},v_0,\psi}=\rho_3\left(|\tilde{v}^i_0\|_{W^{1,2}(\mathbb R_+)},\|\psi^i\|_{C^\beta([0,T])},\|b^*\|_{L^\infty([0,T],L^2(\mathbb R_+))},\|\widetilde\gamma^*\|_{L^\infty([0,T]\times\mathbb R_+)}\right)$.  (We have used the notation   $\rho_{\cdot}(a^*)$  for   $\rho_{\cdot}(a^1,a^2)$.)
\end{proposition}

\begin{proof}
	The bound on $\tilde{u}^1-\tilde{u}^2$ follows immediately from \eqref{eq:u_bound} by linearity. Concerning the bound on $\tilde{v}^1-\tilde{v}^2$, by proceeding as in the a priori bounds \eqref{eq:v_apriori_1} and \eqref{eq:v_apriori_2}, we get
	\begin{align*}
		\|\tilde{v}^1_t-\tilde{v}^2_t\|_{W^{1,2}_x} &\le \|v^1_0-v^2_0\|_{W^{1,2}_x} +A+B,
	\end{align*}
	where the term $A$, resp. $B$, takes into account the contribution of $b^i$, resp. $\widetilde\gamma^i$, $i=1,2$. Precisely, for $A$ we have
	\begin{align*}
		A&= C\int_0^t (t-s)^{-3/4} \left[\|b^1_s\|_{L^2_x}\|\tilde{v}^1_s-\tilde{v}^2_s\|_{W^{1,2}_x}
		+\|b^1_s-b^2_s\|_{L^2_x}\|\tilde{v}^2_s\|_{W^{1,2}_x}\right.\\
		&\quad \quad \quad \left. +\|b^1_s\|_{L^2_x}\|\tilde{u}^1_s-\tilde{u}^2_s\|_{W^{1,2}_x}
		+\|b^1_s-b^2_s\|_{L^2_x}\|\tilde{u}^2_s\|_{W^{1,2}_x} \right] ds.
	\end{align*}
	For $B$ we have, by \eqref{eq:embedding_Linfty_W_1,q}
	\begin{align*}
		B&= C\int_0^t (t-s)^{-1/2} \left[\|\widetilde\gamma^1_s\|_{L^\infty_x}\|\tilde{v}^1_s-\tilde{v}^2_s\|_{L^2_x} + \|\widetilde\gamma^1_s-\widetilde\gamma^2_s\|_{L^2_x}\|\tilde{v}^2\|_{L^\infty_x}\right.\\
		&\quad\quad\quad \left.\|\widetilde\gamma^1_s\|_{L^\infty_x}\|\tilde{u}^1_s-\tilde{u}^2_s\|_{L^2_x} + \|\widetilde\gamma^1_s-\widetilde\gamma^2_s\|_{L^2_x}\|\tilde{u}^2\|_{L^\infty_x} \right] ds\\
		&\le C\int_0^t (t-s)^{-1/2} \left[\|\widetilde\gamma^1_s\|_{L^\infty_x}\|\tilde{v}^1_s-\tilde{v}^2_s\|_{L^2_x} + \|\widetilde\gamma^1_s-\widetilde\gamma^2_s\|_{L^2_x}\|\tilde{v}^2\|_{W^{1,2}_x}\right.\\
		&\quad\quad\quad \left.\|\widetilde\gamma^1_s\|_{L^\infty_x}\|\tilde{u}^1_s-\tilde{u}^2_s\|_{L^2_x} + \|\widetilde\gamma^1_s-\widetilde\gamma^2_s\|_{L^2_x}\|\tilde{u}^2\|_{W^{1,2}_x} \right] ds.
	\end{align*}
	The conclusion derives from a Gronwall-type inequality \cite{2019_webb}(Theorem 3.2), from the bound \eqref{eq:bound_v} on $\tilde{v}^1$ and $\tilde{v}^2$, the bound \eqref{eq:u_bound} on $\tilde{u}^1$, $\tilde{u}^2$ and $\tilde{u}^1-\tilde{u}^2$.
\end{proof}

		\section{Existence, uniqueness and regularity  for the nonlinear system}\label{sec:nonlinear_system}
		
		This section contains the main result of this paper, that is  the well-posedness for the  nonlinear system \eqref{eq:det_nostro_s}-\eqref{eq:stoc_nostro_boundary_cond}. The local well-posedness   is proven by a fixed point argument, while the global one  by an a priori bound.  We follow the proof outlined in \cite{2005_GN_NLA}, but considering $v$ instead of $s$, for the energy bound. 
		In the following, we may let the time horizon $T$ vary in a compact interval $[0,T_{fin}]$.
		\medskip

		Let us gather together all the  assumptions upon the initial and boundary conditions of the dynamical system \eqref{eq:det_nostro_s}-\eqref{eq:stoc_nostro_boundary_cond}. We recall that $\psi$ (as well as $s_0$ and $c_0$) is given; in particular our analysis applies in particular to any random boundary condition whose sample paths satisfy the assumption below.
		\begin{assumption}\label{hp:main_result}
			In the following, $\eta$, $m$, $c_0$ are non negative constants and $C_0\ge m>0$. 
			\begin{itemize}
				\item[B1.] The boundary condition $\psi$ satisfies
				\begin{align*}
					&\psi \in C^\beta([0, T_{fin}]) \text{ for some }1/4<\beta<1/2,\\
					&\psi(0)=0,\\
					&0\le \psi\le \eta.
				\end{align*}
				\item[B2.] The initial conditions $s_0$ and $c_0$ satisfy
				\begin{align*}
					&0\le s_0(x)\le \eta,\quad 0< m\le c_0(x)\le C_0,\quad \forall x \in [0,\infty),\\
					&\text{  $C_0<1$ if $B=-1$},\\
					&s_0\in W^{1,2}_x,\quad C_0-c_0 \in W^{1,2}_x,\\
					&s_0(0)=0.
				\end{align*}
				\item[B3.] If $B=1$, then $\eta<1$. 
			\end{itemize}
		\end{assumption}

  \begin{remark}\label{rmk:pos_varphi} Recalling that $\varphi(c)=1+Bc$, the assumption $C_0<1$ in the case $B=-1$ is equivalent to
  \begin{equation*}
      \varphi_m:=\min_{c\in[0,C_0]}\varphi(c)>0.
  \end{equation*}
   The positivity of $\varphi_m$ is automatically satisfied whenever $B=1$.
   Furthermore, note that in the Assumption \ref{hp:main_result} the hypothesis $\psi(0)=0$ is not strictly necessary, but it simplifies some calculation avoiding more technicality. The last assumption on the initial condition $s_0$ is a right consequence since at the boundary $s_0(0)=\psi(0)$.
        \end{remark}

		\paragraph{The linear system.} The first step of the analysis is  to consider a linearized version of the \eqref{eq:det_nostro_s}-\eqref{eq:det_nostro_c}, where the term $s$ in the ODE for $c$ is replaced by a given term $f$. This makes the equation for $s$ linear.
		
		\begin{assumption}\label{hp_f}
		Let	 us consider the class of Borel function $f:[0,T]\times [0,\infty)\to\mathbb{R}$  which satisfy the following conditions:

		\begin{align}
			\begin{aligned}\label{eq:hp_f}
				&f\in C\left([0,T],L^2(\mathbb R_+)\right) \cap L^2\left([0,T],W^{1,2}(\mathbb R_+)\right);\\
				& \|f\|_{C\left([0,T],L^2(\mathbb R_+)\right) }^2+\|\partial_x f\|_{L^2\left([0,T]\times\mathbb R_+\right)}^2\le K, \quad \mbox{for some }K>0;\\
				&\|f\|_{L^\infty\left([0,T]\times\mathbb R_+\right)}\le \eta;\\
				&f(t,0)=\psi(t)\quad \text{for a.e. }t\in [0,T];\\
				&f\ge 0.
			\end{aligned}
		\end{align}
	\end{assumption}

	\medskip
			
	Let us  consider the following linear PDE-ODE system for $(s,g)$, given a positive Borel function $f$ satisfying Assumption \ref{hp_f}
		\begin{align}
			\partial_t (\varphi(g)s) &= \partial_x(\varphi(g)\partial_xs) -\lambda \varphi(g)sg,\label{eq:lin_PDE_sg_s}\\
			\partial_t g &= -\lambda \varphi(g)fg.\label{eq:lin_PDE_sg_g}
		\end{align}
	In \eqref{eq:lin_PDE_sg_s}-\eqref{eq:lin_PDE_sg_g}	the function $\varphi$ is given by \eqref{eq:phi_c_AB}, the boundary condition   $\psi$is \eqref{eq:stoc_nostro_boundary_cond} and the initial conditions are, for any $x \in (0,\infty)$,
		\begin{align*} 
			s(x,0)&= s_0(x), \qquad  \qquad
			g(x,0)= c_0(x),
		\end{align*}
	where $s_0,c_0$ and $\psi$ satisfy Assumption \ref{hp:main_result}.
		Equation \eqref{eq:lin_PDE_sg_s} for $s$ reads as
		\begin{align}\label{eq:lin_PDE_s_mild}
		\partial_t s = \partial_x^2 s +b_g \partial_xs + \tilde\gamma_g    s,
	\end{align} 
		where $b_g$ is as in \eqref{eq:beta2} and  $\tilde\gamma_g$ is given by \begin{equation}
	\label{def:tilde_gamma}\widetilde{\gamma}_{g} =  -\lambda (1-Bf) g.
\end{equation}.	Note that the solution $g$ of equation  \eqref{eq:lin_PDE_sg_g} admits an  explicit form, that is  
		\begin{align}
			g(x,t) = \displaystyle\frac{c_0(x)}{\varphi(c_0(x))e^{\lambda  \int_0^t f(x,\tau)d\tau} -Bc_0(x)}.\label{eq:g_explicit}
		\end{align}
	
	We seek a solution $(s,g)$ which is a mild solution.
		\begin{definition}[Mild solution for $(s,g)$]
		The couple $(s,g)$, with $s\in L^\infty([0,T],W^{1,2}(\mathbb R_+))\cap L^\infty\left([0,T]\times \mathbb R_+\right)$ and $g\in B_{b}\left([0,T]\times \mathbb R_+\right)$, the space of bounded Borel functions, is \emph{a mild solution} of  PDE system \eqref{eq:lin_PDE_sg_s}-\eqref{eq:lin_PDE_sg_g} if, for every $x \in \mathbb R_+$, $g(\cdot,x)$ solves \eqref{eq:lin_PDE_sg_g}, that is $g$ is given by \eqref{eq:g_explicit} and $s$ is a $L^\infty([0,T],W^{1,2}(\mathbb R_+))$ mild solution of \eqref{eq:lin_PDE_s_mild}. 
	\end{definition}

	\begin{remark}	If $s$ is a mild solution of equation \eqref{eq:lin_PDE_s_mild} and  $s \in C^{1,2}\left( [0,T]\times \mathbb R_+\right)$, the space of function $C^1$ in time and $C^2$ in space, then $s$ satisfies \eqref{eq:lin_PDE_sg_s} for every $(t,x)\in (0,T)\times (0,\infty)$; we say in this case that $s$ is a classical solution of equation \eqref{eq:lin_PDE_sg_s}.
		\end{remark}
		As in Section \ref{sec:uncoupled_equation}, we split the linear solution $s$  of \eqref{eq:lin_PDE_s_mild} as   $s=\tilde{u}+\tilde{v}$, where $\tilde{u}$ solves the heat equation \eqref{eq:u}-\eqref{eq:u_boundary_condition} and $\tilde{v}$ is a $L^\infty([0,T],W^{1,2}(\mathbb R_+))$ mild solution of the equation \eqref{eq:v_lin} with $b=b_g$ and $\tilde\gamma=\tilde\gamma_g$.
		
		\medskip
		
		First of all, we consider some regularity and stability properties for the function $g$ which has an explicit formula given by \eqref{eq:g_explicit}, given the function $f$ satisfying Assumption \ref{hp_f}.
		
		
		\begin{definition}[Good Data]\label{hp:good_data}
			We say that $s_0,c_0,f,\psi$ are \emph{good data} if they satisfy Assumption \ref{hp:main_result} and \ref{hp_f} and if the following regularities are satisfied: there exists an   $0<\alpha<1$ such that $f\in C^{1+\alpha/2,2+\alpha}([0,T]\times [0,+\infty))$, $s_0\in C^{2+\alpha}([0,+\infty))$, $c_0 \in C^{1+\alpha}([0,+\infty))$, $\psi\in C^{1+\alpha/2}([0,T])$ and  
			\begin{align*}
				\dot{\psi}(0) = \partial_x^2 s_0(0)+\frac{\partial_x c_0(0) \partial_x s_0(0)}{\varphi(c_0(0))} +\lambda c_0(0)s_0(0)(f(0,0)-1).
			\end{align*}
		\end{definition}
 
	We have introduced the  definition of good data since  the results in   \cite{2005_GN_NLA} and in \cite{2008_GN_CPDE} are obtained under such  assumptions of good data and then  extended    to a more general setting,  by density arguments. Furthermore, in the case of good data, we get a higher regularity of the solutions, which is useful in some proofs.

		\medskip
		
	The following proposition proposes time uniform bounds for $g$ and its spatial derivative.
		
		\begin{proposition}\label{prop:bd_g}
			Let us assume that $f$ satisfies Assumption  \ref{hp_f}  and that $c_0$ satisfies Assumption \ref{hp:main_result}. Then $0\le g\le C_0$ and 
			\begin{equation}\label{eq:g_continousL2}
				C_0-g\in C([0,T],L^2(\mathbb R_+)).
			\end{equation} Moreover there exists $\kappa>0$, depending on $m$, $C_0$, $\varphi_m$, $\lambda$, $\eta$, $T_{fin}$, but not on $K$, such that, for every $T\in [0,T_{fin}]$, for every $f$ satisfying Assumption  \ref{hp_f}, the following bounds hold   
			\begin{align}
				&\sup_{t\in[0,T]}\|C_0-g\|_{L^2(\mathbb R_+)} \le \kappa \|C_0-c_0\|_{L^2(\mathbb R_+)} +\kappa \|f\|_{C([0,T],L^2(\mathbb R_+))},\label{eq:bound_nonlinear_linear1}\\
				&\sup_{t\in[0,T]}\|\partial_x g\|_{L^2(\mathbb R_+)}^2 \le \kappa \|\partial_x c_0\|_{L^2(\mathbb R_+)}^2 +\kappa \|\partial_x f\|_{L^2([0,T]\times\mathbb R_+)}^2\label{eq:bound_nonlinear_linear2}.
			\end{align}
		\end{proposition}
	
	\begin{proof}

		 As far as condition \eqref{eq:g_continousL2} concerns, we first note that by \eqref{eq:phi_c_AB} and  by Assumptions \ref{hp:main_result} on $c_0$ and $\varphi$, the denominator in \eqref{eq:g_explicit} is bounded from below, hence $g$ is bounded from above. Since $f$ satisfies \eqref{eq:hp_f}, we have, for any $0\le s<t\le T$,  
		\begin{align*}
			|g(t,x)-g(s,x)|&\le C\left|e^{\lambda\int_0^t f(r,x)dr} -e^{\lambda\int_0^s f(r,x)dr}\right|\\
			&\le C\int_s^t f(r,x)dr \cdot \max\{e^{\lambda\int_0^t f(r,x)dr}, e^{\lambda\int_0^s f(r,x)dr}\} \\
			&\le C\int_s^t f(r,x)dr.
		\end{align*}  The constant $C$, changing from line to line depends only upon $f$. We get
		\begin{align*}
			\|g(t,\cdot)-g(s,\cdot)\|_{L^2(\mathbb R_+)}^2 \le C\int_0^\infty \left(C\int_s^t f(r,x)dr\right)^2 dx \le (t-s)\|f\|_{L^2_{[0,T]\times(\mathbb R_+)}},
		\end{align*}
		in particular  $C_0-g$ is in $C([0,T],L^2(\mathbb R_+))$.
			\smallskip
		
			If $f$ and $c_0$ which are good data, in the sense of Definition \ref{hp:good_data}, inequalities \eqref{eq:bound_nonlinear_linear1} and \eqref{eq:bound_nonlinear_linear2} hold true, as shown in   \cite{2008_GN_CPDE}(Proposition 2.1) and \cite{2005_GN_NLA}(Proposition 3.1), respectively. These bounds can be extended to the general case by a density argument. 
			 The proof is complete.
	\end{proof}
		It is possible to establish also some stability results for $g$.
		\begin{proposition}\label{prop:approx_g}
			Let us suppose that $f_1,f_2$ satisfy Assumption  \ref{hp_f}  and that $c_{0,1},c_{0,2}$ satisfy Assumption \ref{hp:main_result}; let $g_1,g_2$ be the corresponding solutions to \eqref{eq:lin_PDE_sg_g}. Then we have, for some $\tilde{K}>0$ (depending on $m$, $C_0$, $\varphi_m$, $\lambda$, $\eta$, $T_{fin}$, $K$), for every $T\in [0,T_{fin}]$,  
			\begin{equation}
            \begin{split}
				\sup_{t\in[0,T]}\|g_1-g_2\|_{L^\infty(\mathbb R_+)} &\le \tilde{K}T\|f_1-f_2\|_{C([0,T],L^2(\mathbb R_+))\cap L^2([0,T],W^{1,2}(\mathbb R_+))} +\tilde{K} \|c_{0,1}-c_{0,2}\|_{L^\infty(\mathbb R_+)},
                \\
				\sup_{t\in[0,T]}\|g_1-g_2\|_{W^{1,2}(\mathbb R_+)} &\le  \tilde{K}T\|f_1-f_2\|_{C([0,T],L^2(\mathbb R_+))\cap L^2([0,T],W^{1,2}(\mathbb R_+))}\label{eq:stability_W12_g}\\
				& +\tilde{K} (1+\|\partial_x c_{0,1}\|_{L^2(\mathbb R_+)}) \|c_{0,1}-c_{0,2}\|_{W^{1,2}(\mathbb R_+)}. 
			  \end{split}
              \end{equation}
		\end{proposition}
		
		\begin{proof}
			It is enough to show the above bounds separately in the case when $c_{0,1}=c_{0,2}$ and in the case when $f_1=f_2$. In the case when $c_{0,1}=c_{0,2}$, the bound follows from \cite{2008_GN_CPDE}(Proposition 2.5) and \cite{2005_GN_NLA}(Proposition 3.4) for $f_1,f_2$ and $c_{0,1}$ good data and a density argument for general $f_1,f_2$ and $c_{0,1}=c_{0,2}$. We consider now the case when $f_1=f_2=:f$. In the following, $C>0$ may denote different constants (depending on $m$, $C_0$, $\varphi_m$, $\lambda$, $T$, $\eta$, $K$). We write, for any solution $g$ to \eqref{eq:lin_PDE_sg_g},
			\begin{align*}
				g(t,x) = \rho(t,x,c_0(x)), \quad \rho(t,x,c)= \frac{c}{\varphi(c)h(t,x)-Bc},\quad h(t,x)=e^{\lambda \int_0^t f(r,x)dr}. 
			\end{align*}
			In view of bounding $g_1-g_2$ and $\partial_x g_1 -\partial_x g_2$, we compute (recall $\varphi(c)=A+Bc$, $B=\pm1$.
			\begin{align*}
				\partial_c\rho(t,x,c) &= \frac{1}{\varphi(c)h(t,x)-Bc} -\frac{(B(h(t,x)-1)c}{(\varphi(c)h(t,x)-Bc)^2},\\
				\partial_x\rho(t,x,c) &= -\frac{c\varphi(c)\partial_x h(t,x)}{(\varphi(c)h(t,x)-Bc)^2},\\
				\partial_c\partial_x\rho(t,x,c) &= \partial_x h(t,x)\left(-\frac{\varphi(c)+Bc}{(\varphi(c)h(t,x)-Bc)^2} +\frac{2c\varphi(c)B(h(t,x)-1)}{(\varphi(c)h(t,x)-Bc)^3}\right),\\
				\partial_c^2\rho(t,x,c) &= -\frac{2B(h(t,x)-1)}{(\varphi(c)h(t,x)-Bc)^2} +\frac{2c(h(t,x)-1)^2}{(\varphi(c)h(t,x)-Bc)^3}.
			\end{align*}
			Since $f$ safisfies \eqref{eq:hp_f}, we have (with a constant $C>0$ independent of $(t,x)$)
			\begin{align*}
				\sup_{c\in [0,C_0]}|\partial_c\rho(t,x,c)|&\le C,\\
				\sup_{c\in [0,C_0]}|\partial_c\partial_x\rho(t,x,c)|&\le C\int_0^t|\partial_x f(r,x)|dr,\\
				\sup_{c\in [0,C_0]}|\partial_c^2\rho(t,x,c)|&\le C.
			\end{align*}
			Now we can bound $g_1-g_2$: since $0\le c_{0,i}\le C_0$, $i=1,2$, we have
			\begin{align*}
				|g_1(t,x)-g_2(t,x)|\le \sup_{c\in [0,C_0]}|\partial_c\rho(t,x,c)||c_{0,1}(x)-c_{0,2}(x)|\le C|c_{0,1}(x)-c_{0,2}(x)|,
			\end{align*}
			which implies, for every $2\le q\le\infty$,
			\begin{align*}
				\sup_{t\in[0,T]}\|g_1-g_2\|_{L^q_x} \le C\|c_{0,1}-c_{0,2}\|_{L^q_x}.
			\end{align*}
			Concerning $\partial_x g_1 -\partial_x g_2$, we have, for $i=1,2$,
			\begin{align*}
				\partial_x g_i(t,x) = \partial_x \rho(t,x,c_{0,i}(x)) +\partial_c \rho(t,x,c_{0,i}(x))\partial_x c_{0,i}(x)
			\end{align*}
			and hence
			\begin{align*}
				|\partial_x g_1(t,x)-\partial_x g_2(t,x)| &\le |\partial_x \rho(t,x,c_{0,1}(x)) -\partial_x \rho(t,x,c_{0,2}(x))|\\
				&\quad +|\partial_c \rho(t,x,c_{0,1}(x)) -\partial_c \rho(t,x,c_{0,2}(x))||\partial_x c_{0,1}(x)|\\
				&\quad +|\partial_c \rho(t,x,c_{0,2}(x))||\partial_x c_{0,1}(x) -\partial_x c_{0,2}(x)|\\
				&\le \sup_{c\in [0,C_0]}|\partial_c\partial_x\rho(t,x,c)||c_{0,1}(x)-c_{0,2}(x)|\\
				&\quad +\sup_{c\in [0,C_0]}|\partial_c^2\rho(t,x,c)||c_{0,1}(x)-c_{0,2}(x)||\partial_x c_{0,1}(x)|\\
				&\quad +\sup_{c\in [0,C_0]}|\partial_c\rho(t,x,c)||\partial_x c_{0,1}(x) -\partial_x c_{0,2}(x)|\\
				&\le C\left(\int_0^T|\partial_x f(r,x)|dr+|\partial_x c_{0,1}(x)|\right) |c_{0,1}(x)-c_{0,2}(x)|\\
				&\quad +C|\partial_x c_{0,1}(x) -\partial_x c_{0,2}(x)|.
			\end{align*}
			Therefore we obtain (by Assumption \ref{hp_f} on $f$)
			\begin{align*}
				\sup_{t\in[0,T]}\|\partial_x g_1-\partial_x g_2\|_{L^2_x}^2 &\le C(\|\partial_x f\|_{L^2_t(L^2_x)}^2 +\|\partial_x c_{0,1}\|_{L^2_x}^2)\|c_{0,1}-c_{0,2}\|_{L^\infty_x}^2 +C\|\partial_x c_{0,1}-\partial_x c_{0,2}\|_{L^2_x}^2\\
				&\le C(1 +\|\partial_x c_{0,1}\|_{L^2_x}^2)\|c_{0,1}-c_{0,2}\|_{W^{1,2}_x}^2,
			\end{align*}
			where in the last line we have used the Sobolev embedding of $W^{1,2}_x$ into $L^\infty_x$. The proof is complete.
		\end{proof}

		Next we show that any function $s$, mild solution  of the linear equation  \eqref{eq:lin_PDE_s_mild}, exists in $[0,T]$, it is pathwise unique and it can be approximated by a sequence of solutions of the same system given  good data approximating the initial conditions, the boundary conditions and the function $f$. The latter property is useful to apply the results proven in \cite{2005_GN_NLA,2008_GN_CPDE}.
		
		\begin{proposition}\label{prop:wellposed_sg_approx}
			Let us consider equation  \eqref{eq:lin_PDE_sg_g}, where  $g$ is the solution of the equation  \eqref{eq:lin_PDE_sg_g}, with  $f$ satisfying Assumption \ref{hp_f}  and with Assumption \ref{hp:main_result} in place.
			\begin{itemize}
				\item[i)]There exists a pathwise unique $L^\infty([0,T],W^{1,2}(\mathbb R_+))$ mild solution $s$ of the equation \eqref{eq:lin_PDE_s_mild}. 
				\item[ii)] If $s_0,c_0,f,\psi$ are good data as in Definition \ref{hp:good_data}, then the corresponding solution $s\in C^{1,2}\left([0,T]\times \mathbb R_+\right)$;  hence, it is a classical solution to \eqref{eq:lin_PDE_sg_s}. 
				\item[iii)]			
			For any mild solution $s$ of the equation \eqref{eq:lin_PDE_s_mild}, there exists a sequence of good data $s_0^n,c_0^n,f^n,\psi^n$, with $f^n$ converging to $f$ in $C([0,T],L^2(\mathbb R_+))\cap L^2([0,T],W^{1,2}(\mathbb R_+))$, with corresponding classical solutions $s^n$ of the equation \eqref{eq:lin_PDE_s_mild}, such that $(s^n)_n$ converges to $s$ in $L^\infty([0,T],W^{1,2}(\mathbb R_+))$, more precisely
			\begin{align}
				\lim\limits_{n\rightarrow \infty} \sup_{t\in [0,T]}\|s^n(t,\cdot)-s(t,\cdot)\|_{W^{1,2}(\mathbb R_+)} =0.\label{eq:approx_good_data}
			\end{align}
			In particular, $s^n$ converges to $s$ uniformly on $[0,T]\times(0,\infty)$.
			\end{itemize}
		\end{proposition}

		\begin{proof}
			Condition i) is a consequence of Corollary \eqref{cor:wellposed_s}, applied with $p=2$, and inequality \eqref{eq:bound_nonlinear_linear1}. Moreover, if $\psi$ is in $C^1$, then well-posedness holds also in $L^\infty_t(W^{1,\infty}_x)$, by Remark \ref{rmk:wellposed_lin_infty}. 
			
			\medskip
			
			The regularity of the solution, stated in condition  ii) in the case of good data, i.e. well-posedness of classical solutions in the class $C^{1+\alpha/2,2+\alpha}_{t,x}$ is given in \cite{2008_GN_CPDE}(Proposition 2.2). Since the classical solution is also in $L^\infty_t(W^{1,\infty}_x)$, it must coincide with the mild solution. 
			
			\medskip
			Let us prove the approximating result iii). Let $s$ be   a mild solution of the equation \eqref{eq:lin_PDE_s_mild}. We take a sequence of good data $s_0^n,c_0^n,f^n,\psi^n$ approximating $s_0,c_0,f,\psi$ in the following sense: $(s_0^n)_n$ converges to $s_0$ in $W^{1,2}_x$, $(C_0-c^n)_n$ converges to $C_0-c$ in $W^{1,2}_x$ and hence in $L^\infty_x$ by Sobolev embedding, $(f^n)_n$ converges to $f$ in $C_t(L^2_x)\cap L^2_t(W^{1,2}_x)$ and $(f^n)_n$ is uniformly bounded in $L^\infty_{t,x}$, $(\psi^n)_n$  converges to $\psi$ in $C^{\beta-\epsilon}_t$, for $\epsilon>0$ given and sufficiently small such that $\beta-\epsilon>1/4$.

			Let $g^n$ be the solution to \eqref{eq:lin_PDE_sg_g} with data $c_0^n$, $f^n$. By \eqref{eq:stability_W12_g}, $(g^n)_n$ converges to $g$ in $L^\infty_t(W^{1,2}_x)$ (and in $C_t(L^2_x)$).
			For any $n\in \mathbb N$, we consider the functions $b^n=b_{g_n}$ and $\tilde\gamma^n=\tilde\gamma_{g_n}$ as in  \eqref{eq:beta2} and \eqref{def:tilde_gamma}. 
			 Therefore, by Assumption \ref{hp:main_result} and Proposition \ref{prop:bd_g}, the family of drifts $(b^n)_n$ converges to $b=b_g$ in $L^\infty_t(L^2_x)$  and the family of coefficients $(\tilde\gamma^n)_n$ is uniformly bounded in $L^\infty_{t,x}$ and converges to $\tilde\gamma=\tilde\gamma_g$ in $C_t(L^2_x)$, where $\gamma$ is given by \eqref{def:tilde_gamma}. Hence, we can apply the stability bound given in Proposition \ref{prop:stability_lin}: calling $s^n$ the solution to \eqref{eq:lin_PDE_sg_g} with good data $s_0^n,c_0^n,f^n,\psi^n$, the family $(s^n)_n$ converges to $s$ in the sense of \eqref{eq:approx_good_data}. By Sobolev embedding $W^{1,2}_x\to C_x$, the convergence is also uniform on $[0,T]\times (0,\infty)$. The proof is complete.
		\end{proof}
		
		\begin{proposition}\label{prop:max_principle}
			Let us assume that  Assumption \ref{hp:main_result} holds and that $f$ satisfies Assumption  \ref{hp_f}. Then, for any mild solution $s$ of the equation \eqref{eq:lin_PDE_s_mild}, we have   
			\begin{align}
				&0\le s(t,x)\le \eta,\quad\quad\qquad \forall (t,x)\in [0,T]\times [0,+\infty),\label{eq:s_less_eta_f}\\
				&\lim_{x\to +\infty}s(t,x) =0,\quad\quad\quad \forall t\in (0,T].\label{eq:s_zero_f}
			\end{align}
			Moreover, let consider the splitted representation $s=\tilde{u}+\tilde{v}$. Then, for some constant $C>0$ (depending only on $T_{fin}$)  we have  
			\begin{align}
				&0\le \tilde{u}(t,x)\le C\|\psi\|_{C([0,T])}\le C \eta,\hspace{1.8cm}\forall (t,x)\in [0,T]\times [0,+\infty),\label{eq:u_bounded_f}\\
				&|\tilde{v}(t,x)|\le C \eta,\hspace{4.65cm} \forall (t,x)\in [0,T]\times [0,+\infty),\label{eq:v_bounded_f}\\
				&\lim_{x\to +\infty}\tilde{u}(t,x) =\lim_{x\to +\infty}\tilde{v}(t,x) =0, \hspace{1.65cm}\forall \, t\in (0,T].\label{eq:u_v_zero_f}
			\end{align}
		\end{proposition}
		
		\begin{proof}
			If $f$, $s_0$, $c_0$ and $\psi$ are good data, \eqref{eq:s_less_eta_f} and \eqref{eq:s_zero_f} are given in \cite{2008_GN_CPDE}(Proposition 2.3). The general case follows by the approximation result in Proposition \ref{prop:wellposed_sg_approx}, namely approximation via solutions with good data. The positivity of $u$ follows from the positivity of $\psi$, by the representation formula \eqref{eq:sol_heat} since  $\partial_x G\le 0$. The uniform bound \eqref{eq:u_bounded_f} and the limit of $u$ in \eqref{eq:u_v_zero_f} follow from the bound \eqref{eq:bound_u_ptwise}. The uniform bound  \eqref{eq:v_bounded_f} follows from \eqref{eq:s_less_eta_f} and \eqref{eq:u_bounded_f}, while the limit of $v$ in \eqref{eq:u_v_zero_f}  follows from \eqref{eq:s_zero_f} and the first of \eqref{eq:u_v_zero_f}.
	
	\end{proof}
	
	Our main original contribution is given by the following proposition  and its corollary \ref{cor:a_priori_bd};   the crucial bound on the splitting variable $\tilde{v}$ that, together with the analogous bound for the splitting variable $\tilde{u}$ , (Proposition \ref{prop:heat_bd}), gives the final bound for the solution $s$ of the linear equation. The corollary that follows  states the generalization of this bound to the nonlinear equation.

		\begin{proposition}[Estimate for the linear system]\label{prop:main_bound_v}
		 Suppose that Assumption \ref{hp:main_result} holds. For every $f$ satisfying Assumption  \ref{hp_f} with a generic $K>0$, let $s$ be the solution of the equation \eqref{eq:lin_PDE_sg_s}, coupled with equation \eqref{eq:lin_PDE_sg_g}. Then,	there exists a constant $\mu \ge 0$, depending on $m$, $C_0$, $\varphi_m$, $\lambda$, $\eta$, $T_{fin}$,  $\|\psi\|_{C^\beta}$, $\|s_0\|_{L^2}$, $\|\partial_x c_0\|_{L^2}$ but not on $K$, such that, for every $T\in [0,T_{fin}]$,  satisfies the following bound   
			\begin{align}
				\sup_{t\in [0,T]}\|s(t,\cdot)\|_{L^2(\mathbb R_+)}^2 +\int_0^T \|\partial_x s(t,\cdot)\|_{L^2(\mathbb R_+)}^2 dt \le \mu +\frac12 \int_0^T \|\partial_x f(t,\cdot)\|_{L^2(\mathbb R_+)}^2 dt.\label{eq:main_bd_1}
			\end{align}
		In particular, for any $f$  satisfying Assumption  \ref{hp_f} with $K\ge 2\mu$,   for any $T\in [0,T_{fin}]$, we have
			\begin{align}
				\sup_{t\in [0,T]}\|s(t,\cdot)\|_{L^2(\mathbb R_+)}^2 +\int_0^T \|\partial_x s(t,\cdot)\|_{L^2(\mathbb R_+)}^2 dt \le K.\label{eq:main_bd_2}
			\end{align}
 
		\end{proposition}
		
		\begin{proof}
			By the approximation result in Proposition \ref{prop:wellposed_sg_approx}, it is enough to show the two bounds \eqref{eq:main_bd_1} and \eqref{eq:main_bd_2} for classical solutions arising from good data $f,s_0,c_0,\psi$. Hence we will now consider such classical solutions.		
			
			Let us remind that we may consider the splitting 
			$s=\tilde{u}+\tilde{v}$,
			where $\tilde{u}$ is solution of \eqref{eq:u}-\eqref{eq:u_boundary_condition} and $\tilde{v}$ is solution of \eqref{eq:v_lin}		
			We first obtain an equation for the  pore concentration for $\varphi(g)\tilde{v}^2$, where $\varphi$ is given by \eqref{eq:phi_c_AB}. By the chain rule, equations \eqref{eq:v_lin} and \eqref{eq:lin_PDE_sg_g}, we get the following  
			\begin{align*}
				\partial_t(\varphi(g)\tilde{v}^2) &= 2\varphi(g)\tilde{v}\partial_t \tilde{v} +\tilde{v}^2\partial_tg \\
				&= 2\tilde{v}\varphi(g)\partial_x^2 \tilde{v} +2\tilde{v}\partial_x(\varphi(g))\partial_x \tilde{v} -2\lambda\varphi(g)g(1-Bf)\tilde{v}^2 +2\partial_x g \partial_x \tilde{u} \tilde{v} -2\lambda\varphi(g)g \tilde{u}(1-Bf)\tilde{v} -\lambda\varphi(g)gf \tilde{v}^2\\
				&= 2\tilde{v}\partial_x(\varphi(g)\partial_x \tilde{v}) -\lambda\varphi(g)g(2+(1-2B)f)\tilde{v}^2 +2\partial_xg \partial_x \tilde{u} \tilde{v} -2\lambda\varphi(g)g\tilde{u}(1-Bf)\tilde{v}.
			\end{align*}
		Let us integrate   over  $(0,t)\times(0,+\infty)$. By means of the integration by parts in $x$ in the first term at the right hand side, since by Proposition \ref{prop:max_principle} $v(t,0)=0$ and $v(t,+\infty)=0$), we get
			\begin{align*}
				&\int_0^\infty \varphi(g_t)\tilde{v}_t^2 dx -\int_0^\infty \varphi(g_0)v_0^2 dx +2\int_0^t\int_0^\infty \varphi(g)(\partial_x \tilde{v})^2 dxdr\\
				&= -\int_0^t\int_0^\infty \lambda\varphi(g)g(2+(1-2B)f)\tilde{v}^2 dxdr +\int_0^t\int_0^\infty 2 \tilde{v}\partial_xg \partial_x \tilde{u}  dxdr +\int_0^t\int_0^\infty -2\lambda\varphi(g)g  (1-Bf)\tilde{u}\tilde{v} dxdr \\
				&=: R_1+R_2+R_3.
			\end{align*}
			By Proposition \ref{prop:max_principle} and Assumption \ref{hp:main_result}, we have $\tilde{u}\ge 0$, $f\ge 0$ and, for $B=1$, $f\le \eta<1$ and, for $B=-1$ $2+3f>0 $. Hence 
			$$ R_1 \le 0.$$   As the  $ R_2$ term concerns, Proposition \ref{prop:max_principle} gives  the $L^\infty$ uniform bound on $\tilde{v}$, inequality \eqref{eq:bound_nonlinear_linear1}  gives the $L^2$ bound on $\partial_x g$, while    the $L^2$ bound on $\partial_x u$ derives by Proposition \ref{prop:heat_bd}, since $\beta>1/4$; hence,  for $\epsilon>0$ to be determined later and the constants $C_\epsilon$ and $C$,   depending only on $\epsilon$ and only on $T_{fin}$, respectively,  we get what follows
			\begin{align*}
				 R_2&=\int_0^t\int_0^\infty 2\partial_xg \partial_x \tilde{u}\tilde{v}   dxdr\\
				&\le C_\epsilon \int_0^T \|v\partial_x \tilde{u}\|_{L^2_x}^2 dr +\epsilon \int_0^T\|\partial_x g\|_{L^2_x}^2 dr\\
				&\le C_\epsilon T\|\tilde{v}\|_{L^\infty_{t,x}}^2 \|\partial_x \tilde{u}\|_{L^\infty_t(L^2_x)}^2 +\epsilon \,T\,\|\partial_xg\|_{L^\infty_t(L^2_x)}^2\\
				&\le C C_\epsilon T \eta^2\|\psi\|_{C^\beta}^2 +\epsilon\, \kappa T \left(\|\partial_x c_0\|_{L^2_x}^2 +  \|\partial_x f\|_{L^2_t(L^2_x)}^2\right).
			\end{align*}
			For the term $ R_3$, we use the $L^\infty$ bound on $g$ given by \eqref{eq:bound_nonlinear_linear1}, the $L^1$ bound \eqref{eq:Lp_norm_u} on $u$ and the bound $|\psi|\le \eta$, the uniform bound on $v$ by Proposition \ref{prop:max_principle} and Assumption \ref{hp:main_result} and get
			\begin{align*}
				 R_3&= \int_0^t\int_0^\infty -2\lambda\varphi(g)g(1-Bf)\tilde{u}\tilde{v} dxdr\\
				&\le 2\lambda T\|g\|_{L^\infty_x}(A+B\|g\|_{L^\infty_x})\|\tilde{v}\|_{L^\infty_{t,x}}\|\tilde{u}\|_{L^\infty_t(L^1_x)}\\
				&\le C\lambda T\eta^2C_0(A+BC_0).
			\end{align*}
			Hence, by means of the estimates of the terms $R_1-R_3$ and  by considering the lower bound for $\varphi$ as in condition B2 of the Assumption \ref{hp:main_result}, we get the following space estimate in $L^2 $ for $\tilde{v}$ and its derivative
			\begin{align*}
				\int_0^\infty \tilde{v}_t^2 dx  +2\int_0^t\int_0^\infty (\partial_x \tilde{v})^2 dxdr &\le \varphi_m^{-1}\int_0^\infty \varphi(g_t)\tilde{v}_t^2 dx  +2\varphi_m^{-1}\int_0^t\int_0^\infty \varphi(g)(\partial_x \tilde{v})^2 dxdr\\
				&\le \varphi_m^{-1}C_0\|s_0\|_{L^2_x}^2 +C C_\epsilon \varphi_m^{-1} T \eta^2\|\psi\|_{C^\beta}^2 +\epsilon\varphi_m^{-1}\kappa T \|\partial_x c_0\|_{L^2_x}^2\\
				&\quad +C\lambda T \eta^2C_0(A+BC_0) +\epsilon \varphi_m^{-1}\kappa T \|\partial_x f\|_{L^2_t(L^2_x)}^2\\
				&=:A_\epsilon+\epsilon \varphi_m^{-1}\kappa T \|\partial_x f\|_{L^2_t(L^2_x)}^2.
			\end{align*}
			If we combine the latter estimate with the inequality \eqref{eq:u_bound}   for $\tilde{u}$, then, for some constant $\bar{C}$ depending only on $T_{fin}$, we get the following for $s=\tilde{u}+\tilde{v}$, 
			\begin{align*}
				\sup_{t\in [0,T]}\int_0^\infty s_t^2 dx  +\int_0^T\int_0^\infty (\partial_x s)^2 dxdr &\le 2\sup_{t\in [0,T]}\int_0^\infty \tilde{u}_t^2 dx  +2\int_0^T\int_0^\infty (\partial_x \tilde{u})^2 dxdr +2\sup_{t\in [0,T]}\int_0^\infty \tilde{v}_t^2 dx \\
				&\quad +2\int_0^T\int_0^\infty (\partial_x \tilde{v})^2 dxdr\\
				&\le \bar{C}\|\psi\|_{C^\beta}^2 + 3A_\epsilon+3\epsilon \varphi_m^{-1}\kappa T \|\partial_x f\|_{L^2_t(L^2_x)}^2.
			\end{align*}
			Now, let us  take $\epsilon>0$, sufficiently small  that $3 \epsilon \varphi_m^{-1}\kappa T_{fin}\le 1/2$; we get
			\begin{align*}
				\sup_{t\in [0,T]}\int_0^\infty s_t^2 dx  +\int_0^T\int_0^\infty (\partial_x s)^2 dxdr \le  \bar{C}\|\psi\|_{C^\beta}^2 +3A_{\epsilon} +\frac12 \|\partial_x f\|_{L^2_t(L^2_x)}^2.
			\end{align*}	
			So, inequality  \eqref{eq:main_bd_1} is proven, with $\mu=\bar{C}\|\psi\|_{C^\beta}^2 +3A_{\epsilon}$. If we  take the parameter $K$ in \eqref{eq:hp_f}  such that $K\ge 2 \mu$, we conclude that
			\begin{align*}
				\sup_{t\in [0,T]}\int_0^\infty s_t^2 dx  +\int_0^T\int_0^\infty (\partial_x s)^2 dxdr\le \frac{K}{2}+\frac{K}{2} =K.
			\end{align*}
			Hence,   inequality \eqref{eq:main_bd_2} is proven.
		\end{proof}

\medskip

	\begin{remark}\label{rmk_Natalini1}
	Proposition \ref{prop:main_bound_v} is the main point where the irregularity of the boundary condition $\psi$ does not let us use the results already proven in \cite{2008_GN_CPDE}. Precisely, in \cite{2008_GN_CPDE}(Proposition 2.4), the estimate on $s$ is obtained by multiplying the equation for $\varphi(g)s$ morally with $s-\psi$ and integrating by parts in time and space; in this procedure the time derivative of $\psi$ appears, something which cannot be controlled in our context. Here instead, in a certain sense we leave the irregularity due to $\psi$ into the $u$ term and consider the equation for $v$ instead, which has smooth (actually zero) boundary condition.
\end{remark}

\begin{remark}
	Note that, unlike \cite{2008_GN_CPDE}(Proposition 2.4), the bound \eqref{eq:main_bd_1} in Proposition \ref{prop:main_bound_v} on the solution $s$ to the linear equation \eqref{eq:lin_PDE_sg_s} depends on the $C([0,T],L^2(\mathbb R_+))\cap L^2([0,T],W^{1,2}(\mathbb R_+))$ norm of $f$. The reason for this dependence is due to the appearance of the term $\partial_x g$, which can only be controlled by the $C([0,T],L^2(\mathbb R_+))\cap L^2([0,T],W^{1,2}(\mathbb R_+))$ norm of $u$. However, we can make this dependence small, that is, with a multiplicative constant less or equal than one half. This is enough to get the a priori estimate in Proposition \ref{cor:a_priori_bd} in the case of  the nonlinear equation \eqref{eq:s2}-\eqref{eq:gamma2}.
\end{remark}

		The next result deals with the contraction property of the operator  that maps  any $f$ in \eqref{eq:hp_f} into the solution of the linear equation,   for small time  $T$.
		
		\begin{proposition}\label{prop:contraction}
			 Suppose that Assumption \ref{hp:main_result} holds. There exists $\mu>0$ (as in Proposition \ref{prop:main_bound_v}) such that, for every $K\ge 2\mu$, we have: there exist $T\in [0,T_{fin}]$, $L<1$ (all depending on $m$, $C_0$, $\varphi_m$, $\lambda$, $\eta$, $T_{fin}$, \ $\|\psi\|_{C^\beta}$, $\|s_0\|_{L^2}$, $\|\partial_x c_0\|_{L^2}$ and $K$) such, that, for every $f^1$, $f^2$ satisfying \eqref{eq:hp_f} with $K$, calling $s^1$, $s^2$ the corresponding solutions, we have
			\begin{align*}
				\sup_{t\in[0,T]}\|s^1_t-s^2_t\|_{L^2(\mathbb R_+)}^2 +\int_0^T \|\partial_x s^1_t -\partial_x s^2_t\|_{L^2(\mathbb R_+)}^2 dt \le L \left( \sup_{t\in[0,T]}\|f^1_t-f^2_t\|_{L^2(\mathbb R_+)}^2 +\int_0^T \|\partial_x f^1_t -\partial_x f^2_t\|_{L^2(\mathbb R_+)}^2 dt \right).
			\end{align*}
		\end{proposition}
		
		\begin{proof}
			If $s^1$ and $s^2$ are solutions arising from good data, then one can repeat the arguments in the proof of \cite{2008_GN_CPDE}(Proposition 2.6), by replacing the constant $K_\lambda$ in \cite{2008_GN_CPDE} with the constant $K\ge 2\mu$, where $\mu$ is as in Proposition \ref{prop:main_bound_v}. The general case follows by the approximation result in Proposition \ref{prop:wellposed_sg_approx}.
		\end{proof}
	
	\paragraph{The nonlinear system} Let us focus again on  the nonlinear PDE \eqref{eq:s2}-\eqref{eq:gamma2}. First of all we consider a corollary of Proposition \ref{prop:main_bound_v}.
	\begin{proposition}[Estimate for the nonlinear system]\label{cor:a_priori_bd}
	 Suppose that Assumption \ref{hp:main_result} holds. If $(s,c)$ is a mild solution to the nonlinear PDE \eqref{eq:s2}-\eqref{eq:gamma2}. Then there exists a constant $\mu\ge0$, depending on $m$, $C_0$, $\varphi_m$, $\lambda$, $\eta$, $T_{fin}$,  $\|\psi\|_{C^\beta}$, $\|s_0\|_{L^2}$, $\|\partial_x c_0\|_{L^2}$ but not on $K$, such that, for every $T\in [0,T_{fin}]$, if $(s,c)$ is a mild solution to the nonlinear PDE \eqref{eq:s2}-\eqref{eq:gamma2}, we have   
	\begin{align*}
		\sup_{t\in[0,T]}\|s(t,\cdot)\|_{L^2(\mathbb R_+)}^2 +\int_0^T \|\partial_x s(t,\cdot)\|_{L^2(\mathbb R_+)}^2 dt \le 2\mu.
	\end{align*}
\end{proposition}
\begin{proof}
	If $(s,c)$ solves the nonlinear PDE \eqref{eq:s2}-\eqref{eq:gamma2}, one can consider $f=s$ and the conclusion follows from the bound \eqref{eq:main_bd_1} of   Proposition \ref{prop:main_bound_v}. .
\end{proof}

\begin{remark}	We note that if $(s,c)$ is a mild solution to \eqref{eq:s2}-\eqref{eq:gamma2}, then by Remark \ref{rmk:time_continuity_lin} $s\in C([0,T];L^2(\R^2))$.
	Furthermore,  we have that $(s,c)$ is a mild solution to \eqref{eq:s2}-\eqref{eq:gamma2} if and only if $(s,c)$ is a mild solution to \eqref{eq:lin_PDE_sg_s}-\eqref{eq:lin_PDE_sg_g} with $f=s$.
\end{remark}
The global  existence of a pathwise unique mild solution for the nonlinear equation is proven in the following theorem.
		\begin{theorem}\label{thm:main}
			Let us consider system \eqref{eq:s2}-\eqref{eq:gamma2} on $[0,T_{fin}]\times \mathbb R_+$. Suppose that Assumption \ref{hp:main_result}  holds.
			For any $T\in [0,T_{fin}]$, there exists a pathwise unique mild solution $(s,c)$, in the sense of Definition \ref{def:mild_solution_s_nonlin}.
		\end{theorem}

		\begin{proof}
			Let $\chi_{K,T}$ denote the space of functions $f$ satisfying Assumption \ref{hp_f}, with $K$ and $T$ as in Proposition \ref{prop:main_bound_v} and Proposition \ref{prop:contraction}; it turns out that $\chi_{K,T}$ is a complete metric space endowed with the norm $C_t(L^2_x)\cap L^2_t(W^{1,2}_x)$. 
			
			\medskip
			
			We consider the solution map of the liner system $$F:\chi_{K,T} \to \chi_{K,T}, \qquad F(f)=s,$$ where, given a function $f$, $s$ is the first component of the solution $(s,g)$  to the linear system \eqref{eq:lin_PDE_sg_s}-\eqref{eq:lin_PDE_sg_g} driven by $f$. By Proposition \ref{prop:max_principle} and Proposition \ref{prop:main_bound_v}, the operator $F$ is well-defined and it is a contraction on $\chi_{K,T}$ (for small $T$ and large $K$), due to   Proposition \ref{prop:contraction}. Hence, we get existence and uniqueness of a fixed point of $F$, that is existence and uniqueness for \eqref{eq:s2}-\eqref{eq:gamma2}, for small $T>0$. For general $T>0$, we can use a strategy which is similar to the one followed in \cite{2008_GN_CPDE}. We take $K=2\mu$, with  $\mu$ as in Proposition \ref{prop:main_bound_v} and Proposition \ref{cor:a_priori_bd} and define
			\begin{align*}
				T^* =\sup\left\{t\in [0,T]\mid \exists ! \text{ mild solution $(s,c)$ to } \eqref{eq:s2}-\eqref{eq:gamma2}\right\}.
			\end{align*}
		
			First of all, we show that $(s,c)$ can be extended as a mild solution up to $T^*$ included and $s$ is such that it satisfies that  $\sup_{t\in [0,T^*]}\|s_t\|_{W^{1,2}_x}<\infty$. In order to do this, we take $f=s$, extended to $t=T^*$ for example taking $f_{T^*}=0$. With such $f$, $c$ can be extended up to $T^*$ included by the explicit formula \eqref{eq:g_explicit} and the bounds \eqref{eq:bound_nonlinear_linear1}-\eqref{eq:bound_nonlinear_linear2} still hold. By Corollary \ref{cor:wellposed_s}, there exists a unique mild solution $\tilde{s}$ on $[0,T^*]$ to \eqref{eq:lin_PDE_sg_s}, in particular $\tilde{s}$ is in $C_t(L^2_x)$ and satisfies $\sup_{t\in [0,T^*]}\|\tilde{s}_t\|_{W^{1,2}_x}<\infty$. By uniqueness of \eqref{eq:lin_PDE_sg_s}, $s$ and $\tilde{s}$ must coincide on $[0,T^*)$. Hence $s$ can be extended to a solution on $[0,T^*]$ by setting $s_{T^*}=\tilde{s}_{T^*}$ and it verifies $\sup_{t\in [0,T^*]}\|s_t\|_{W^{1,2}_x}<\infty$.
			
			Finally, we show that $T^*=T$. By contradiction, if $ T^*<T$ then we would like to take $s_{T^*}$ (which is in $W^{1,2}_x$) and $c_{T^*}$ as new initial conditions and apply the local well-posedness result to extend the unique solution for $T>T^*$. Unfortunately here we cannot use exactly this argument, because $s_{T^*}=\psi_{T^*}$ does not need to be zero and so we are outside Assumption \ref{hp:main_result}. However, we can consider a modified fixed point problem as follows.  We take the space $\tilde{\chi}^{T^*}_{K,T}$ of functions $f$ satisfying Assumption \ref{hp_f} and such that $f=s$ (the unique solution of \eqref{eq:s2}-\eqref{eq:gamma2} on $[0,T^*]\times [0,\infty)$ and we still consider the map $F$, restricted on $\tilde{\chi}^{T^*}_{K,T}$, such that $F(f)$ is the solution to the (first equation of) the linear system \eqref{eq:lin_PDE_sg_s}-\eqref{eq:lin_PDE_sg_g} driven by $f$. The proof of Proposition \ref{prop:contraction} (which is the proof of \cite{2008_GN_CPDE}(Proposition 2.6)) does not use the fact that $\psi(0)=0$ and so one can repeat the proof for $f$ in $\tilde{\chi}^{T^*}_{K,T}$ and get the following fact, for $K$ large enough and $T-T^*>0$ small enough, for some $L<1$: for every $f^1$, $f^2$ in $\tilde{\chi}^{T^*}_{K,T}$, calling $s^1=F(f^1)$, $s^2=F(f^2)$,
			\begin{align*}
				\sup_{t\in[T,T^*]}\|s^1_t-s^2_t\|_{L^2_x}^2 +\int_{T^*}^T \|\partial_x s^1_t -\partial_x s^2_t\|_{L^2_x}^2 dt \le L \left( \sup_{t\in[T^*,T]}\|f^1_t-f^2_t\|_{L^2_x}^2 +\int_{T^*}^T \|\partial_x f^1_t -\partial_x f^2_t\|_{L^2_x}^2 dt \right).
			\end{align*}
			Hence, for large $K$ and small $T-T^*$, $F$ is a contraction on $\tilde{\chi}^{T^*}_{K,T}$ and therefore it has a unique fixed point, that is there exists a unique mild solution $(s,c)$ on $[0,T^*+T]$. But this is in contradiction with the definition of $T^*$. Hence $T^*=T$. The proof is complete.
		\end{proof}

\paragraph{Stability with respect to the boundary condition.} In this paragraph we show a stability bound of the mild solution $(s,c)$ with respect to the boundary condition $\psi$. The stability result implies in particular that, for random adapted $\psi$, the (random) solution $(s,c)$ is adapted as well.

We start with stability for the linear system \eqref{eq:lin_PDE_sg_s}-\eqref{eq:lin_PDE_sg_g}:

\begin{proposition}
    Assume that $s_0$ and $c_0$ satisfy Assumption \ref{hp:main_result}. Let $\psi_1$, $\psi_2$ be two boundary conditions satisfying Assumptions \ref{hp:main_result}, $f_1$, $f_2$ satisfying Assumption \ref{hp_f}, let $(s_1,g_1)$, $(s_2,g_2)$ be the corresponding mild solutions to \eqref{eq:lin_PDE_sg_s}-\eqref{eq:lin_PDE_sg_g}. Then there exists $\bar{C}>0$, depending on $m$, $C_0$, $\varphi_m$, $\lambda$, $\eta$, $T_{fin}$, $\|\partial_xc_0\|_{L^2_x}$, $\|s_i\|_{L^2_t(W^{1,2}_x)}$ and $\|f_i\|_{L^2_t(W^{1,2}_x)}$, $i=1,2$ (but not directly on $\psi_i$ or on $f_i$) such that
    \begin{align}\label{eq:stability_linear}
        \sup_{t\in [0,T]}\|s_1-s_2\|_{L^2_x}^2 +\int_0^T\|\partial_x s_1 -\partial_x s_2\|_{L^2_x}^2 dt \le \bar{C}\|\psi_1-\psi_2\|_{C^\beta_t}^2 +\bar{C}\int_0^T \|s_1-s_2\|_{L^2_x}^2dt +\frac12\int_0^T\|f_1-f_2\|_{W^{1,2}_x}^2 dt.
    \end{align}
\end{proposition}

The proof follows the lines of the proof of \cite{2008_GN_CPDE}(Proposition 2.6), but exploits the splitting strategy to deal with a non-zero boundary condition for the difference $s_1-s_2$. We recall the following inequality (consequence of Sobolev-Gagliardo-Nirenberg, or of \eqref{eq:prop_embedding} and \cite{2018_Lunardi}, Example 1.3.8):
\begin{align*}
\|v\|_{L^\infty_x}^2 \le C_1\|v\|_{W^{1/2,2}_x}^2 \le C_2 \|v\|_{L^2_x} \|v\|_{W^{1,2}_x}.
\end{align*}

\begin{proof}
   By a density argument (and Proposition \ref{prop:wellposed_sg_approx}), it is enough to show the result for $s_0$, $c_0$, $f_i$ and $\psi_i$ good data as in Definition \ref{hp:good_data}, for which $s_1$ and $s_2$ are $C^{1,2}([0,T]\times \R_+)$. We call $f=f_1-f_2$, $\psi=\psi_1-\psi_2$, $s=s_1-s_2$, $g=g_1-g_2$. We split $s=u+v$, where $u$ is the solution to the heat equation \eqref{eq:u}-\eqref{eq:u_initial_condition}-\eqref{eq:u_boundary_condition} (with $\psi$ as boundary condition) and $v=v_1-v_2$ (where $s_i=u_i+v_i$) satisfies   
   \begin{align}
   \begin{aligned}\label{eq:diff_eq}
       \partial_t(\varphi(g_1)v^2) &=
       2v\partial_x(\varphi(g_1)\partial_x v)
       +2v B\partial_x g_1 \partial_x u \\
       &\quad -2\lambda\varphi(g_1)g_1v^2(1-Bf_1)
       -2\lambda\varphi(g_1)g_1uv(1-Bf_1) \\
       &\quad +2v\varphi(g_1)\partial_x s_2 \left(\frac{B\partial_x g_1}{\varphi(g_1)}-\frac{B\partial_x g_2}{\varphi(g_2)}\right)
       -2v\lambda\varphi(g_1)s_2(g_1(1-Bf_1)-g_2(1-Bf_2)) \\
       &\quad -B\lambda \varphi(g_1)f_1g_1v^2 \\
       &=: I+T_1+T_2+T_3+T_4+T_5+T_6.
    \end{aligned}
   \end{align}
We integrate each term in time and space (in the following, $C$ represents a generic positive constant, which depends on $m$, $C_0$, $\varphi_m$, $\lambda$, $\eta$, $T_{fin}$ and may change from one line to another). Since $v(t,0)=0$, the term $I$ gives
\begin{align*}
    \int_0^T\int_{\R_+} I dxdt = -2\int_0^T\int_{\R_+} \varphi(g_1)|\partial_x v|^2 dxdt \le -C\int_0^T\int_{\R_+} |\partial_x v|^2 dxdt,
\end{align*}
where we have used that $\varphi(g_1)$ is bounded from below (see Remark \ref{rmk:pos_varphi}). Recalling the bound \eqref{eq:u_bound} on $u$, the term $T_1$ gives, for fixed $\delta>0$ and $C_\delta>0$ (possibly changing from one line to another),
\begin{align*}
    \left|\int_0^T\int_{\R_+} T_1 dxdt\right|
    &\le C\int_0^T\|v\|_{L^\infty_x}^2 dt
    +C(\|\partial_x g_1\|_{L^\infty_t(L^2_x)}^2+\|\partial_x g_2\|_{L^\infty_t(L^2_x)}^2) \|\partial_x u\|_{L^\infty_t(L^2_x)}^2 \\
    &\le C_\delta\int_0^T\|v\|_{L^2_x}^2 dt
    +\delta\int_0^T\|\partial_xv\|_{L^2_x}^2 dt
    +C(\|\partial_xf_1\|_{L^2_{t,x}}^2+\|\partial_xf_2\|_{L^2_{t,x}}^2+\|\partial_xc_0\|_{L^2_x}^2) \|\psi\|_{C^\beta_t}^2,
\end{align*}
where we have used the bound \eqref{eq:bound_nonlinear_linear2} on $g_i$.
The term $T_2$ gives
\begin{align*}
    \left|\int_0^T\int_{\R_+} T_2 dxdt\right|
    &\le C(1+\|g_1\|_{L^\infty_{t,x}}^2)(1+\|f_1\|_{L^\infty_{t,x}}) \int_0^T \|v\|_{L^2_x}^2 dt \\
    &\le C\int_0^T \|v\|_{L^2_x}^2 dt,
\end{align*}
where we have used that $g_1$ and $f_1$ are bounded respectively by $C_0$ and $\eta$ (see Proposition \ref{prop:bd_g}). Similarly, using \eqref{eq:u_bound}, the term $T_3$ gives
\begin{align*}
    \left|\int_0^T\int_{\R_+} T_3 dxdt\right|
    &\le C\int_0^T \|v\|_{L^2_x}^2 dt +C\int_0^T \|u\|_{L^2_x}^2 dt \\
    &\le C\int_0^T \|v\|_{L^2_x}^2 dt +C\|\psi\|_{C^\beta_t}^2,
\end{align*}
and the term $T_6$ gives
\begin{align*}
    \left|\int_0^T\int_{\R_+} T_6 dxdt\right|
    &\le C\int_0^T \|v\|_{L^2_x}^2 dt.
\end{align*}
The terms $T_4$ gives, for fixed $\delta>0$ and constants $\tilde{c}>0$ (depending on $m$, $C_0$, $\varphi_m$, $\lambda$, $\eta$, $T_{fin}$, but independent of $\delta$) and $C_\delta>0$ (possibly changing from one line to another),
\begin{align*}
    \left|\int_0^T\int_{\R_+} T_4 dxdt\right|
    &\le C_\delta\int_0^T \|v\|_{L^\infty_x}^2 dt +\delta \|\partial_x s_2\|_{L^2_t(L^2_x)}^2(1+\|g_1\|_{L^\infty_{t,x}}^2)\varphi_m^{-4}\cdot\\
    &\quad \quad \cdot \left[\|g_1-g_2\|_{L^\infty_t(L^\infty_x)}^2\|\partial_x g_1\|_{L^\infty_t(L^2_x)}^2 +\|\partial_xg_1-\partial_x g_2\|_{L^\infty_t(L^2_x)}^2\|g_2\|_{L^\infty_t(L^\infty_x)}^2\right] dt\\
    &\le C_\delta\int_0^T\|v\|_{L^2_x}^2 dt +\delta\int_0^T\|\partial_xv\|_{L^2_x}^2 dt \\
    &\quad +\delta \tilde{c}\|\partial_x s_2\|_{L^2_t(L^2_x)}^2 \|f_1-f_2\|_{L^2_t(W^{1,2}_x)}^2 (1+\|\partial_x f_1\|_{L^2_t(L^2_x)}^2+\|\partial_x c_0\|_{L^2_x}^2),
\end{align*}
where we recall that $\varphi(g_1)$ is bounded from below (by $\varphi_m>0$) and the Propositions \ref{prop:bd_g} and \ref{prop:approx_g} on $g$ and $g_i$. Similarly, the term $T_5$ gives
\begin{align*}
    \left|\int_0^T\int_{\R_+} T_5 dxdt\right|
    &\le C_\delta\int_0^T \|v\|_{L^\infty_x}^2 dt +\delta \| s_2\|_{L^\infty_{t,x}}^2(1+\|g_1\|_{L^\infty_{t,x}}^2)\cdot\\
    &\quad \quad \cdot \left[\|g_1-g_2\|_{L^\infty_t(L^\infty_x)}^2\|f_1\|_{L^2_t(L^2_x)}^2 +\|f_1-f_2\|_{L^2_t(L^2_x)}^2\|g_2\|_{L^\infty_t(L^\infty_x)}^2\right] dt \\
    &\le C_\delta\int_0^T\|v\|_{L^2_x}^2 dt +\delta\int_0^T\|\partial_xv\|_{L^2_x}^2 dt \\
    &\quad +\delta \tilde{c}\|f_1-f_2\|_{L^2_t(W^{1,2}_x)}^2 (1+\|f_1\|_{L^2_t(L^2_x)}^2+\|f_2\|_{L^2_t(L^2_x)}^2).
\end{align*}
where we have used that $s_i$ is bounded respectively by $\eta$ (see Proposition \ref{prop:max_principle}) and again the Propositions \ref{prop:bd_g} and \ref{prop:approx_g} on $g_i$. Integrating \eqref{eq:diff_eq} in time and space, we get, for constants $\bar{c}>0$, $\bar{C}>0$ (depending on $m$, $C_0$, $\varphi_m$, $\lambda$, $\eta$, $T_{fin}$, $\|\partial_xc_0\|_{L^2_x}$, $\|s_i\|_{L^2_t(W^{1,2}_x)}$ and $\|f_i\|_{L^2_t(W^{1,2}_x)}$),
\begin{align*}
    \|v(T)\|_{L^2_x}^2 +\int_0^T\|\partial_x v\|_{L^2_x}^2 dt \le C\|\psi\|_{C^\beta_t}^2 +C_\delta\int_0^T \|v\|_{L^2_x}^2dt  +\bar{c}\delta\int_0^T \|\partial_xv\|_{L^2_x}^2dt +\bar{c}\delta\int_0^T\|f\|_{W^{1,2}_x}^2 dt.
\end{align*}
Choosing $\delta$ with $\tilde{c}\delta\le 1/8$, combining the above bound on $v$ with $\|u\|_{L^\infty_t(W^{1,2}_x)}\le C\|\psi\|_{C^\beta_t}$ from \eqref{eq:u_bound}, we obtain for $s=u+v$ (possibly renaming $\bar{C}$)
\begin{align*}
    \|s(T)\|_{L^2_x}^2 +\int_0^T\|\partial_x s\|_{L^2_x}^2 dt \le \bar{C}\|\psi\|_{C^\beta_t}^2 +\bar{C}\int_0^T \|s\|_{L^2_x}^2dt +\frac12\int_0^T\|f\|_{W^{1,2}_x}^2 dt.
\end{align*}
The proof is complete.
\end{proof}

As a consequence, we get the stability bound for the nonlinear system \eqref{eq:s2}-\eqref{eq:gamma2}:

\begin{proposition}
    Let $\psi_1$, $\psi_2$ be two boundary conditions satisfying Assumptions \ref{hp:main_result}, let $(s_1,g_1)$, $(s_2,g_2)$ be the corresponding mild solutions to the nonlinear system \eqref{eq:s2}-\eqref{eq:gamma2} (as from Theorem \ref{thm:main}). Then there exists $\bar{C}>0$, depending on $m$, $C_0$, $\varphi_m$, $\lambda$, $\eta$, $T_{fin}$, $\|\partial_xc_0\|_{L^2_x}$ and $\|s_i\|_{L^2_t(W^{1,2}_x)}$ $\|f_i\|_{L^2_t(W^{1,2}_x)}$, $i=1,2$ (but not directly on $\psi_i$ or on $s_i$) such that
    \begin{align*}
        \sup_{t\in [0,T_{fin}]}\|s_1-s_2\|_{L^2_x}^2 +\frac12\int_0^{T_{fin}}\|\partial_x s_1 -\partial_x s_2\|_{L^2_x}^2 dt \le \bar{C}\|\psi_1-\psi_2\|_{C^\beta_t}^2.
    \end{align*}
\end{proposition}

\begin{proof}
    The bound is a consequence of \eqref{eq:stability_linear}, taking $s_i=f_i$ the solution to the nonlinear system \eqref{eq:s2}-\eqref{eq:gamma2} and applying Gr\"onwall lemma.
\end{proof}

Finally, we conclude adaptedness when $\psi$ is an adapted stochastic process:

\begin{corollary}
    Let $\beta>1/4$, let $(\Omega,\mathcal{A},(\mathcal{F}_t)_t,P)$ be a filtered probability space satisfying the standard assumption and let $\psi$ be a $(\mathcal{F}_t)_t$-progressively measurable process with paths in $C^\beta_t$ satisfying Assumption \ref{hp:main_result} (for possibly random $\eta>0$). Let $(s,c)$ be the corresponding mild solution to the nonlinear system \eqref{eq:s2}-\eqref{eq:gamma2}. Then $s$ is $(\mathcal{F}_t)_t$-progressively measurable as $L^2(\R_+)$-valued process.
\end{corollary}

\begin{proof}
    For every $T\le T_{fin}$, the map $[0,T]\times \Omega\ni (t,\omega)\mapsto s^\omega(t) \in L^2(\R_+)$ is the composition of the two maps $(t,\omega)\mapsto (t,\gamma(\omega)) \in [0,T]\times C^\beta([0,T])$, which is measurable from $\mathcal{B}([0,T])\otimes \mathcal{F}_T$ to $\mathcal{B}(C^\beta([0,T])$, and $(t,\gamma)\mapsto s^{\gamma(\omega)}(t)$, which is measurable from $\mathcal{B}(C^\beta([0,T])$ to $\mathcal{B}(L^2(\R_+))$. This shows progressive measurability.
\end{proof}

Clearly progressively measurability of $c$ holds as well (for example in the space of functions $g$ with $C_0-g\in L^2(\R_+)$), as a consequence of progressively measurability of $s$ and the explicit expression \eqref{eq:df_mild_solution_c} of $c$ in terms of $s$.

\section{Appendix}

For convenience of the reader here we recall some classical results of functional analysis related to fractional Sobolev spaces with some embedding results \cite{2012_dinezza_palatucci_valdinoci,2018_Lunardi,1978_triebel}. Furthermore, we give the basic definition related to the interpolation theory  \cite{1988_bennet_sharpley,2018_Lunardi,1964_Lions,1978_triebel}.  
\medskip
  
Let us remark the fact that, as in the classical case with integer exponents, the space $W^{k^\prime+\alpha^\prime,p}$ is
continuously embedded in $W^{k+\alpha,p}$, with  $k+\alpha \le k^\prime+\alpha^\prime$, $k, k^\prime$ integers and $\alpha,\alpha^\prime \in (0,1)$. The result holds in the limit case $\alpha=1$ \cite{2012_dinezza_palatucci_valdinoci}.
\begin{proposition}\cite{2012_dinezza_palatucci_valdinoci}\label{prop:Sobolev_embedding}
	Let $p\in [1,\infty]$, $k,k^\prime\in \mathbb N$, and $\alpha,\alpha^\prime\in (0,1)$, with $k+\alpha \le k^\prime+\alpha^\prime$. Then
	$$\|f\|_{W^{k+\alpha,p}(\mathbb R) }\le C \|f\|_{W^{k^\prime+\alpha^\prime,p}(\mathbb R) }.$$
	In particular, 
	$$W^{k^\prime+\alpha^\prime,p}(\mathbb R)  \subseteq W^{k +\alpha ,p}(\mathbb R).$$
	The previous hold also in the case $k=k^\prime=0$. Furthermore,
	$$W^{ 1,p}(\mathbb R)  \subseteq W^{\alpha,p}(\mathbb R).$$
\end{proposition}
The Sobolev embedding results may be extended by considering different $L^p$ spaces, for various exponents $p$, provided some relations among the exponents.
\begin{proposition}\cite{2018_Lunardi,1978_triebel}
	Let $p,q\in [1,\infty]$, $s,s^\prime\in \mathbb R_+$, with $p\le q $ and $s^\prime \le s$. If 
	\begin{equation}\label{eq:prop_embedding_parameters}
		s  -\frac{1}{p}\ge s^\prime -\frac{1}{q}, 
	\end{equation} then  
	\begin{equation}\label{eq:prop_embedding} 
		W^{s,p}(\mathbb R)  \subseteq W^{s^\prime ,q}(\mathbb R).\end{equation} 
\end{proposition}
All the previous results are applied also to $\mathbb R_+$.

\section*{Acknowledgments}
The authors are members of GNAMPA (Gruppo Nazionale per l’Analisi Matematica, la Probabilità e le loro Applicazioni) of the Italian Istituto Nazionale di Alta Matematica (INdAM).

   \end{document}